\documentclass[a4paper,12pt]{amsart}
\usepackage{amssymb}

\DeclareMathSymbol{\subsetneqq}{\mathbin}{AMSb}{36}

\newcommand{\R}{\mathbb{R}}

\newcommand{\N}{\mathbb{N}}

\newcommand{\C}{\mathbb{C}}

\newcommand{\cL}{\mathcal L}
\newcommand{\dint}{\displaystyle\int}

\newcommand{\dsup}{\displaystyle\sup}

\newcommand{\dinf}{\displaystyle\inf}

\newcommand{\beq}{\begin{eqnarray}}
\newcommand{\eeq}{\end{eqnarray}}
\newcommand{\bq}{\begin{equation}}
\newcommand{\eq}{\end{equation}}
\newcommand{\beqn}{\begin{eqnarray*}}
\newcommand{\eeqn}{\end{eqnarray*}}
\newcommand{\bex}{\begin{exo}}
\newcommand{\eex}{\end{exo}}
\newcommand{\ben}{\begin{enumerate}}
\newcommand{\een}{\end{enumerate}}

\newtheorem{th1}{{\bf Theorem}}[section]
\newtheorem{thm}[th1]{{\bf Theorem}}
\newtheorem{lem}[th1]{{\bf Lemma}}
\newtheorem{prop}[th1]{{\bf Proposition}}

\newtheorem{rem}[th1]{\bf Remark}

\newtheorem{defi}[th1]{\bf Definition}

\author[S.~Ibrahim]{Slim Ibrahim}
\address{Department of Mathematics and Statistics,\\University of Victoria\\
 PO Box 3060 STN CSC\\   Victoria, BC, V8P 5C3\\ Canada}
\email{\sl ibrahim@math.uvic.ca}
\urladdr{ http://www.math.uvic.ca/~ibrahim/}
\thanks{{\sf S. Ibrahim} is partially supported by NSERC\# 371637-2009 grant and
a start up fund from the University of Victoria.}
\author[R. Jrad]{Rym Jrad}
\address{University Tunis ElManar,
Faculty of Sciences of Tunis, Department of Mathematics.}
\email{\sl rym.jrad@gmail.com}

\author[M. Majdoub]{Mohamed Majdoub}
\address{University Tunis ElManar,
Faculty of Sciences of Tunis, Department of Mathematics.}
\email{\sl mohamed.majdoub@fst.rnu.tn}

\author[T. Saanouni]{Tarek Saanouni}
\address{University Tunis ElManar,
Faculty of Sciences of Tunis, Department of Mathematics.}
\email{\sl tarek.saanouni@ipeiem.rnu.tn}
\thanks{{\sf R. Jrad, M. Majdoub \& T. Saanouni} are grateful to the Laboratory of
PDE and Applications at the Faculty of Sciences of Tunis.}


\title[Well posedness and unconditional uniqueness ...]{Well posedness and unconditional non uniqueness for a 2D semilinear heat equation%The Cauchy problem for a two dimensional semi-linear heat equation with exponential growth
}

\date{\today}
\begin{document}
\begin{abstract}
We investigate the initial value problem for a semilinear heat equation with exponential-growth nonlinearity in two space dimension. First, we prove the local existence and unconditional uniqueness of solutions in the Sobolev space $H^1(\R^2)$. The uniqueness part is non trivial although it follows Brezis-Cazenave's proof \cite{Br} in the case of monomial nonlinearity in dimension $d\geq3$. Next, %Following Caffarelli-Vasseur \cite{cv},
we show that in the defocusing case our solution is bounded, and therefore exists for all time. In the focusing case, we prove that any solution with negative energy blows up in finite time. Lastly, we show that the unconditional result is lost
once we slightly enlarge the Sobolev space $H^1(\R^2)$. The proof consists in constructing a singular stationnary solution that will gain some regularity when it serves as initial data in the heat equation. The Orlicz space appears to be appropriate for this result since, in this case, the potential term is only an integrable function.
\end{abstract}
%@@@@@@@@@@@@@@@@@@@@@@@@@@@@@@@@@@@@%@@@@@@@@@@@@@@@@@@@@@@@@@@@@@@@@@@@@%@@@@@@@@@@@@@@

%@@@@@@@@@@@@@@@@@@@@@@@@@@@@@@@@@@@@%@@@@@@@@@@@@@@@@@@@@@@@@@@@@@@@@@@@@%@@@@@@@@@@@@@@

\maketitle

%\tableofcontents
\vspace{ 1\baselineskip}
\renewcommand{\theequation}{\thesection.\arabic{equation}}

\section{Introduction}
Consider the initial value problem for a semilinear heat equation
\begin{equation}
\label{eq1}
\left\{
\begin{matrix}
\partial_t u=\Delta u+ f(u)\\
u(0)= u_0\\
\end{matrix}
\right.
\end{equation}
where $u(t,x) : \R^+\times\R^d\to\R$, $d\geq2$ and $f\in{\mathcal C}^1(\R,\R)$ is a given function satisfying $f(0)=0$. The Cauchy problem \eqref{eq1} has been extensively studied in the scale of Lebesgue spaces $L^q$, especially for polynomial type nonlinearities i.e
\begin{equation}
\label{mono}
f(u):= \pm |u|^{\gamma-1}u,\quad \gamma>1.
\end{equation}
In such a case, observe that the equation enjoys an interesting property of scaling invariance
\begin{equation}
\label{scaling}
u_\lambda(t,x):=\lambda^{2/\gamma}\,u(\lambda^2 t, \lambda x), \quad \lambda>0
\end{equation}
i.e. if $u$ solves \eqref{mono} then also does $u_\lambda$. The Lebesgue space $L^{q_c}(\R^d)$ with index $q_c:=\frac{d(\gamma-1)}{2}$ is also the only one invariant under the same scaling \eqref{scaling}. This property defines a sort of  trichotomy in the dynamic of solutions of \eqref{mono}, and basically one can notice the following three different regimes:\\

{\sc {The subcritical case i.e. $q> q_c\geq1$:}} Weissler in \cite{Ws1} proved the existence of a unique solution $u\in {\mathcal C}([0,T); L^q(\R^d))\cap L_{loc}^\infty(]0,T]; L^\infty(\R^d))$. Later on, Brezis-Cazenave \cite{Br} proved the unconditional uniqueness of Weissler's solutions.\footnote{Uniqueness in the natural space where solutions exist, namely ${\mathcal C}(L^q)$.}\\

{ {\sc The critical case} i.e. $q= q_c$ and $d\geq 3$:} There are two sub-cases:
\begin{itemize}
\item
If $q_c>\gamma+1$, then we have local wellposedness of the Cauchy problem where the existence is also due to  Weissler \cite{Ws1} and the {unconditional uniqueness} to Brezis-Cazenave \cite{Br}.\\
\item
If $q= q_c=\gamma+1$ or equivalently $q={\frac{d}{d-2}}$ and $\gamma-1={\frac{2}{d-2}}$ (double critical or energy critical case\footnote{Observe that in such a case, the potential energy term is finite.}): Weissler \cite{Ws2} proved the conditional wellposedness. When the underlying space is the unit ball of $\R^d$, Ni-Sacks \cite{NiSa} showed that the unconditional uniqueness fails. This result was extended to the whole space by E. Terraneo \cite{T} for suitable intial data. See also \cite{MaTerr} for general initial data.
\end{itemize}

\vspace{0.5cm}

{\sc{The supercritical case i.e. $q<q_c$:}} there are indications that there exists no (local) solution in any reasonable weak sense ({\rm cf.} \cite{Br, Ws1, Ws2}). Moreover, it is known that uniqueness is lost for the initial data $u_0=0$ and for $1+\frac{1}{d}<\gamma<\frac{d+2}{d-2}$, see Haraux-Weissler \cite{HW}.\\

The way in constructing solutions consists in using a fixed point argument in suitable spaces where the free solution lives and the nonlinear terms can be estimated using the heat regularizing properties. Note that the solution can be written as
$$
u(t)={\rm e}^{t\Delta}u_0+M(u)(t),
$$
where the integral operator $M(u)(t):=\int_0^t{\rm e}^{(t-s)\Delta}f(u(s))\;ds$. This operator behaves differently in the sub and critical cases. It is clearly continuous in ${\mathcal C}([0,T);L^q(\R^d))$ when the nonlinearity is subcritical, while it is discontinuous in the critical case (see \cite{NiSa} for more details).

In the energy critical case, the nice idea of Ni and Sacks \cite{NiSa} to prove the non-uniqueness is constructive and based on the fact that the Poisson equation does not regularize as much as the heat equation when the source term is only an integrable function. In the energy critical case, the potential term $|u|^{\gamma-1}u\in L^\infty_t(L^1)$. So, Ni and Sacks constructed a singular stationary solution in the punctured unit ball. The singularity holds only at the center of the ball and is weak enough to extend the singular solution (in the distributional sense) to the whole ball. Then, they constructed a local solution which will immediately enjoy a smoothing effect that the stationary singular solution will never have. This makes the two solutions different and the unconditional non-uniqueness immediately follows. Let us mention that the well posedness in Sobolev and Besov spaces was investigated in \cite{Rib, MI}.\\

In two space dimension, observe that the energy\footnote{i.e. $q_c=\gamma+1$.} scaling index  $q_c={\frac{d}{d-2}}$ becomes infinite. So any power nonlinearity $1<\gamma<\infty$ is subcritical in the sense that one can always choose a Lebesgue space $L^q$ (other than $L^\infty$) where one can prove the well-posedness for the Cauchy problem \eqref{eq2}. However, when taking an infinite polynomial e.g exponential nonlinearity, the only Lebesgue space in which Weissler's result is applicable is $L^\infty$. To this extent, the Cauchy problem \eqref{eq2} is always subcritical in $L^\infty$ and one can wonder if there is any notion of criticality in two space dimension. The loss of the scaling property for inhomogeneous nonlinearities also does not help in having any insight toward an answer.\\

 The aim of this paper is to show that in $2D$, a kind of trichotomy (similar to the one described above in higher dimensions ) can still be defined. It is based on the topology of the initial data. More precisely, consider the Cauchy problem
\begin{equation}
\label{eq2}
\left\{
\begin{matrix}
\partial_t u-\Delta u=\pm u({\rm e}^{ {u}^2}-1)\quad\mbox{in}\quad \R^2\\\\
u(0)= u_0\,.\\
\end{matrix}
\right.
\end{equation}
Our first goal in this paper is to study whether or not there exists local/global solution to the Cauchy problem \eqref{eq2} when the data is no longer in $L^\infty$. \\

First, observe that for an exponential nonlinearity, the largest Lebesgue type space in which the equation is meaningful in the distributional sense is of Orlicz kind. In this respect, Ruf and Terraneo \cite{RT} showed a local existence result for {\it small} initial data in Orlicz space in four space dimension. In what follows, we will focus our attention only to the case $d=2$.\\
Recall that the Sobolev space $H^{1}(\R^2)$ is embedded in all Lebesgue spaces $L^p$ for every $2\leq p<\infty$ but not in $L^\infty$. The optimal (critical) Sobolev embedding is known to be
 \bq
 \label{Embed}
 H^{1}(\R^2)\hookrightarrow {\cL}(\R^2),
 \eq
 where ${\cL}(\R^2)$ is the Orlicz space associated to the function $\phi(s)={\rm e}^{s^2}-1$ (see next section for the precise definition of Orlicz space). The embedding \eqref{Embed}  is sharp within the context of Orlicz spaces in the sense that the target
space $\cL$ cannot be replaced by an essentially smaller Orlicz space. In addition, note that initial data in $H^1(\R^2)$ or ${\cL}(\R^2)$ are not necessarily bounded functions and therefore, Weissler's result for the local existence of solutions does not apply.\\
Second, we will show that we have a ``good" $H^1$ theory for the Cauchy problem \eqref{eq2} i.e. finite time/global existence of solutions (depending on the sign of the nonlinearity), and unconditional uniqueness.\\
Finally, once we enlarge a little bit the space of data by taking them in ${\cL}(\R^2)$, we show that the unconditional uniqueness is lost. This space is quite natural for this result because the potential term is only an $L^1$ function.\\
Our results show that even though there is no a scaling property for this problem, a sort of trichotomy analogous to the one described in higher dimension can still be defined. It is based on the topology of the initial data. In a forthcoming paper, we show the non-existence of solutions of the Cauchy problem \eqref{eq2} if the initial data is in the Sobolev space $H^s(\R^2)$ with $s<1$.\\

This paper is organized as follows. In the next section, we state our main results. In Section 3, we recall some basic definitions and auxiliary lemmas. The fourth section deals with the $H^1$ regularity regime. Section 5 is devoted to the Orlicz regularity data.\\

Finally, we mention that $C$ will be used to denote a constant which may vary from line to line.
We also use $A\lesssim B$ to denote an estimate of the form $A\leq C B$
for some absolute constant $C$ and $A\approx B$ if $A\lesssim B$ and $B\lesssim A$.

 %%%%%%%%%%%%%%%%%%%%%%%%%%%%%%%%%%%%%%%%%%%%%%%%%%%%%%%%%%%%%%%%%%%%%%%%%%%
 \section{Main results}
%%%%%%%%%%%%%%%%%%%%%%%%%%%%%%%%%%%%%%%%%%%%%%%%%%%%%%%
First, we prove that without any restriction on the size of the initial data, the Cauchy problem \eqref{eq2} is locally well-posed in the Sobolev space $H^1(\R^2)$. To do so, we use a standard fixed point argument. The uniqueness part is non trivial and follows the steps of Brezis-Cazenave's proof \cite{Br} in the case of monomial nonlinearity in dimension $d\geq3$. Following Caffarelli-Vasseur \cite{cv} and using the energy estimate, we prove that in the defocusing case our solution is bounded. Hence by a standard blow-up
criterion (see for example \cite{Br}) the solution extends to a global one. Proceeding in the same way as in \cite{Tan}, we show that in the focusing case that any solution to \eqref{eq2} with an initial data with a negative energy blows up in finite time.
Recall that the energy is given by
$$
J(u(t)):=\frac12\|\nabla u(t)\|_{L^2(\R^2)}^2-\int_{\R^2}F(u(t))\;dx, \quad\mbox{with}\quad F(u)=\int_0^u\,f(v)\,dv\,.
$$
Our first main result can be stated as follows.
 \begin{thm}
\label{MainH1}
Let $u_0\in H^1(\R^2)$.
\begin{enumerate}
\item[1)]There exists a {\it unique} $u$ solution to \eqref{eq2} in ${\mathcal C}([0,T]; H^1)$.
\item[2)] If $f(u)=-u({\rm e}^{  u^2}-1)$, then the (above) solution is global.
\item[3)] If $f(u)=u({\rm e}^{u^2}-1)$, then a data $u_0\neq0$ with $J(u_0)\leq0$ gives a unique solution blowing up in finite time.
\end{enumerate}
\end{thm}
\begin{rem}
The first assertion of the above Theorem remains true for $f(u)=\pm u{\rm e}^{  u^2}$, and the second one also extends to the case $f(u)=- u{\rm e}^{  u^2}$. This means that we only need to remove the quadratic term from the nonlinearity only for the blow up result.
\end{rem}

The previous Theorem shows that the $H^1$ regularity supports well the exponential nonlinearity. That is why we have obtained a ``good" $H^1$-theory. Now, we enlarge that space a little bit so that \eqref{eq2} is still meaningful in the distributional sense, and we investigate the well posedness of \eqref{eq2} in the Orlicz space $\mathcal L$ which is larger than the usual Sobolev space $H^1(\R^2)$. First, we improve the result of Ruf and Terraneo \cite{RT} by showing the local existence of solutions. Second, we give a non uniqueness result in the Orlicz space based on a construction of a singular solution to the associated elliptic problem.

\begin{thm}\label{MainOrlicz}
Let $B_1$ be the unit ball of $\R^2$.
There exists infinitely many $u_0\in \mathcal{L}(B_1)$ such that the Cauchy problem

\begin{eqnarray}
\label{eqDomain}
\left\{
\begin{matrix}
\partial_t u=\Delta u+u({\rm e}^{u^2}-1)\quad \mbox{in}\quad B_1\\\\
u_{t=0}=u_0\quad \mbox{in}\quad B_1\\\\
u_{|\partial B_1}=0\quad \mbox{for}\quad t>0
\end{matrix}
\right.
\end{eqnarray}
with data $u_0$ has at least two (distinct) solutions.
\end{thm}
To prove the above Theorem, we first construct infinitely many stationary, non-negative and radially symmetric singular solutions $Q$. We then show that they all belong to the Orlicz space. This implies that $f(Q)\in L^1$, and thus the elliptic regularity does not reach $L^\infty$. Second, using a such singular solution as an initial data in \eqref{eqDomain}, we can construct a solution to the heat equation in $C([0,T),\mathcal{L}(B_1))$. For that, we split the initial data into a smooth part (localized away from the singularity), and small and singular part (well localized near the singularity). We easily construct a local smooth solution with the smooth initial data, and then by a perturbation argument, we construct a solution to the problem with the localized small singular data. Using a parabolic regularization result due to Brezis-Cazenave \cite{Br}, we show that this solution also enjoys a smoothing effect and is in $L^{\infty}([0,T);\mathcal{L}(B_1))\cap L^\infty_{loc}((0,T); L^\infty(B_1))$.

\begin{rem}
In order to prove a nonuniqueness result when the underlying space is $\R^2$, one can either construct a singular solution on the whole space using Pacard's method, or extend to $\R^2$ the singular solutions that we construct on the punctured ball. This latter singular function will obviously not solve the stationary problem. Thus, one need to construct a singular solution to the heat equation with that data. In this paper, we elect to restrict our selves to the ball and not the whole space $\R^2$.
\end{rem}

%%%%%%%%%%%%%%%%%%%%%%%%%%%%%%%%%%%%%%%%%%%%%%%%%%%%%%%%%%%%%%%%%%%%%%%%%%%%%%%%%%%%%%%%%%%%%%%%%%%%%%%%%%%%%%%%%%%%%%%%

\section{Background material}
%%%%%%%%%%%%%%%%%%%%%%%%%%%%%%%%%%%%%%%%%%%%%%%%%%%%%%%%%%%%%%%%%%%%%%%%%%%%%%%%%%%%%%%%%%%%%%%%%%%%%%%%%%%%%%%%%%%%%%%%%
In this section we will fix the notation, state the basic definitions and recall some known and useful tools. First we recall the standard smoothing effect (see for example \cite{Br}).
\begin{lem}\label{smw}
There exists a positive constant $C$ such that for all $1\leq\beta\leq\gamma\leq\infty$, we have
\begin{equation}\label{smoth}
\|{\rm e}^{t\Delta}\varphi\|_{L^{\gamma}}\leq\frac{C}{t^{\frac{1}{\beta}-\frac{1}{\gamma}}}\|\varphi\|_{L^{\beta}},\quad\forall t>0,\forall\varphi\in L^{\beta}(\R^2)
\end{equation}
where ${\rm e}^{t\Delta}\varphi:=K_t*\varphi=\frac{1}{4\pi t}{\rm e}^{-\frac{|\,.\,|^2}{4t}}*\varphi$.
\end{lem}
Using Young and H\"older inequalities and the precedent Lemma with the following integral formula
$$u(t)={\rm e}^{t\Delta}u_0+\int_0^t{\rm e}^{(t-s)\Delta}\left(\partial_t u-\Delta u\right) (s)\,ds$$
we deduce the following estimates
\begin{prop}\label{nrg}
\begin{equation}\label{enrg}
\sup_{t\in [0,T]}\|u(t,.)\|_{H^1({\mathbb{R}}^{2})}\leq C\Big(\|u(t_0,.)\|_{H^1({\mathbb{R}}^{2})}+\|\partial_t u-\Delta u\|_{L^{1}([0,T],H^{1}({\mathbb{R}^{2})})}\Big).
\end{equation}
\begin{equation}\label{enrg1}
\sup_{t\in [0,T]}\|u(t,.)\|_{L^{\infty}({\mathbb{R}}^{2})}\leq C\Big(\|u(t_0,.)\|_{L^{\infty}({\mathbb{R}}^{2})}+\|\partial_t u-\Delta u\|_{L^{1}([0,T],L^{\infty}({\mathbb{R}^{2})})}\Big).
\end{equation}
\end{prop}
We recall the following nonlinear estimates which are consequence of the mean
value theorem and the convexity of the exponential function. See \cite{Ib2, Col.I}.
\begin{lem}\label{f}
For any $\varepsilon >0$ there exists $C_{\varepsilon}>0$ such that
\begin{equation}\label{f1}
|f(U_1)-f(U_2)|\leq C_{\varepsilon} |U_1-U_2|\sum_{i=1}^2\Big(e^{ (1+\varepsilon)U_i^2}-1\Big).
\end{equation}
\begin{equation}\label{f2}
|f^{'}(U_1)-f^{'}(U_2)|\leq C_{\varepsilon} |U_1-U_2|\sum_{i=1}^2\Big({\rm e}^{ 2(1+\varepsilon)U_i^2}-1\Big)^{1/2}.
\end{equation}
\end{lem}
In order to control the nonlinear part in $L^{1}_{t}(H^1_{x})$, we will use the following Moser-Trudinger  inequality \cite{Ad,Mo,Tr}.
\begin{prop}\label{prop3}
Let $\alpha\in (0,4\pi)$, a constant $C_{\alpha}$ exists such that
for all $u\in H^{1}({\mathbb{R}}^{2})$ satisfying $\|\nabla u\|_{L^{2}({\mathbb{R}}^{2})}\leq 1$, we have
\begin{equation}\label{eq17}
\int_{{\mathbb{R}}^{2}}\Big({\rm e}^{\alpha |u(x)|^{2}}-1\Big)dx\leq C_{\alpha}\|u\|^{2}_{L^{2}({\mathbb{R}}^{2})}.\end{equation}
Moreover, \eqref{eq17} is false if $\alpha\geq 4\pi$.
\end{prop}
Let us mention that $\alpha=4\pi$ becomes admissible if we require $\|u\|_{H^{1}({\mathbb{R}}^{2})}\leq 1$ rather than
$\|\nabla u\|_{L^{2}({\mathbb{R}}^{2})}\leq 1$. Precisely
\begin{equation}\label{mos}
\displaystyle\sup_{\|u\|_{H^{1}({\mathbb{R}}^{2})}\leq 1}\int_{{\mathbb{R}}^{2}}\Big({\rm e}^{4\pi |u(x)|^{2}}-1\Big)dx<\infty
\end{equation}
and this is false for $\alpha>4\pi$. See \cite{Ru} for more details.\\

Let us now introduce the so-called Orlicz spaces on $\R^d$ and some
related basic facts.
\begin{defi}\label{deforl}\quad\\
Let $\phi : \R^+\to\R^+$ be a convex increasing function such that
$$
\phi(0)=0=\lim_{s\to 0^+}\,\phi(s),\quad
\lim_{s\to\infty}\,\phi(s)=\infty.
$$
We say that a measurable function $u : \R^d\to\C$ belongs to
$L^\phi$ if there exists $\lambda>0$ such that
$$
\dint_{\R^d}\,\phi\left(\frac{|u(x)|}{\lambda}\right)\,dx<\infty.
$$
Then, we denote \bq \label{norm}
\|u\|_{L^\phi}=\dinf\,\left\{\,\lambda>0,\quad\dint_{\R^d}\,\phi\left(\frac{|u(x)|}{\lambda}\right)\,dx\leq
1\,\right\}. \eq
\end{defi}
It is easy to check that $L^\phi$ is a $\C$-vectorial space and
$\|\cdot\|_{L^\phi}$ is a norm. Moreover, we have the following
properties.\\ \noindent$\bullet$ For $\phi(s)=s^p,\, 1\leq
p<\infty$,  $L^\phi$ is nothing else than the  Lebesgue space
$L^p$.\\ \noindent$\bullet$ For  $\phi_\alpha(s)={\rm e}^{\alpha
s^2}-1$, with $\alpha>0$,  we claim  that
$L^{\phi_\alpha}=L^{\phi_1}.$ It is actually a direct consequence
of Definition \ref{deforl}.\\ \noindent$\bullet$ We may replace in
\eqref{norm} the number $1$ by any positive constant. This change
the norm $\|\cdot\|_{L^\phi}$ to an equivalent norm.\\
\noindent$\bullet$  For $u\in L^\phi$ with $A:=\|u\|_{L^\phi}>0$,
we have the following property \bq \label{norm1}
\left\{\,\lambda>0,\quad\dint_{\R^d}\,\phi\left(\frac{|u(x)|}{\lambda}\right)\,dx\leq
1\,\right\}=[A, \infty[\,. \eq

In what follows we shall fix  $d=2$, $\phi(s)={\rm e}^{ s^2}-1$ and
denote the Orlicz space $L^\phi$ by ${\mathcal L}$ endowed with the
norm $\|\cdot\|_{\mathcal L}$. It is easy to see that ${\mathcal L}\hookrightarrow L^p$
for every $2\leq p<\infty$. The 2D critical Sobolev embedding in Orlicz space ${\mathcal L}$
states as follows:
$$
H^{1}(\R^2)\hookrightarrow {\cL}(\R^2)\,.
$$
We recall some elementary properties about Orlicz spaces (see for example \cite{HMN,RT}).
\begin{prop} \label{or1} We have\\
{\bf a)} $\left(L^\phi, \|\cdot\|_{L^\phi}\right)$ is a Banach space.\\
{\bf b)} $L^1\cap L^\infty\subset L^\phi\subset L^1+L^\infty$.\\
{\bf c)} There exists a positive real number $\kappa_{\phi,d}$ such that if $T : L^1\to L^1$ with norm $M_1$ and $T : L^\infty\to
L^\infty$ with norm $M_\infty$,  then $T : L^\phi\to L^\phi$ with
norm $M\leq\kappa_{\phi,n}\dsup(M_1,M_\infty)$.\\
{\bf d)} For any $p\geq 2$, $\mathcal L(\R^2)\subset L^p(\R^2)$ and we have
$$\|u\|_{L^p(\R^2)}\leq \Big(\Gamma(\frac{p}{2}+1)\Big)^{\frac{1}{p}}\|u\|_{\mathcal L(\R^2)}\, ,$$
where $\Gamma(x):=\int_0^{+\infty}t^{x-1}{\rm e}^{-t}dt$.
\end{prop}

Now, we give some technical results which will be  useful later. The following lemma is classical (see for example Proposition 4.2 of \cite{Tay}) but the proof seems to be new.
\begin{lem}\label{int}
Let $u\in H^1(\R^2)$. Then for any $ \alpha>0$ and $1\leq q<\infty$,
$$
{\rm e}^{\alpha u^2}-1\in L^q(\R^2).
$$
\end{lem}
\begin{proof}[Proof of Lemma \ref{int}] Without loss of generalitye, we may assume that $\alpha=q=1$ and $u$ is radial. First, let us observe that thanks to the following well known radial estimate
 $$
 |u(r)|\leq
\frac{C}{\sqrt{r}}\,\|u\|_{H^1},
$$
we obtain for any $a>0$,
$$
\int_{|x|\geq a}\,\left({\rm e}^{|u(x)|^2}-1\right)\,dx\leq \int_{|x|\geq a}\,|u(x)|^2\,{\rm e}^{|u(x)|^2}\,dx\leq\,{\rm e}^{\frac{C}{a}\|u\|_{H^1}^2}\, \|u\|_{L^2}^2<\infty\,.
$$

Therefore, to conclude the proof it is sufficient to show that for suitable $a>0$, we have
\bq
\label{near0}
\int_0^a\, {\rm e}^{u^2(r)}\, r\,dr<\infty\,.
\eq
For $a>0$ and $0<r<a$, write
\beqn
|u(r)-u(a)|&=&\Big|\int_r^a\,\sqrt{s}\, u'(s)\,\frac{ds}{\sqrt{s}}\Big|\\
&\leq&\frac{1}{\sqrt{2\pi}}\,\|\nabla u\|_{L^2(|x|<a)}\,\left(-\log(\frac{r}{a})\right)^{1/2}\,.
\eeqn
Choosing $a>0$ small enough such that
$$
\|\nabla u\|_{L^2(|x|<a)}^2<2\pi,
$$
and witting
\beqn
{\rm e}^{u^2(r)}\, r&\lesssim& {\rm e}^{2(u(r)-u(a))^2}\, r\\
&\lesssim& r^{1-\beta}, \quad \beta:=\frac{\|\nabla u\|_{L^2(|x|<a)}^2}{\pi},
\eeqn
we end up with \eqref{near0}.
\end{proof}
\begin{prop}\label{int1}
Let $u\in \mathcal C([0,T];H^1(\R^2))$ for some $T>0$. Then
$${\rm e}^{u^2}-1\,\,\in \,\,\mathcal C([0,T];L^1(\R^2)).$$
\end{prop}
\begin{proof}[Proof of Proposition \ref{int1}]
Let $t\in [0,T]$ and $(t_n)$ be a sequence in $(0,T)$ such that $t_n \rightarrow t $. Denote by $u_n:=u(t_n)$ and $u=u(t)$.
We will prove that
$$
{\rm e}^{u_n^2}-1\rightarrow {\rm e}^{u^2}-1\quad\mbox{in}\quad L^1(\R^2).
$$
Set $v_n:=u_n-u$. Clearly, we have
$$
{\rm e}^{u_n^2}-{\rm e}^{u^2}={\rm e}^{u^2}\Big(({\rm e}^{v_n^2}-1)({\rm e}^{2v_nu}-1)+({\rm e}^{2v_nu}-1)+({\rm e}^{v_n^2}-1)\Big).
$$
Hence
\begin{eqnarray}
\|{\rm e}^{u_n^2}-{\rm e}^{u^2}\|_{L^1}
&\lesssim&\|{\rm e}^{u^2}-1\|_{L^{2}}\,\|({\rm e}^{v_n^2}-1)({\rm e}^{2v_nu}-1)+({\rm e}^{2v_nu}-1)+({\rm e}^{v_n^2}-1)\|_{L^2}\nonumber\\
&+&\|({\rm e}^{v_n^2}-1)({\rm e}^{2v_nu}-1)+({\rm e}^{2v_nu}-1)+({\rm e}^{v_n^2}-1)\|_{L^{1}},\label{1}
\end{eqnarray}
and by Moser-Trudinger  inequality we have
\begin{equation}\label{2}
\displaystyle\lim_{n}\|{\rm e}^{v_n^2}-1\|_{L^2}=0.
\end{equation}
Now, it is sufficient to prove that
$$
\lim_n\|{\rm e}^{2|uv_n|}-1\|_{L^1}= 0.
$$
From Proposition \ref{or1}, recall that for every $p\geq 2$, we have
$$\|\cdot\|_{L^p}\leq(\Gamma(\frac{p}{2}+1))^{\frac{1}{p}}\|\cdot\|_{\mathcal L}.$$
Thus, by H\"older inequality we have the estimate
\begin{eqnarray*}
\|{\rm e}^{2|uv_n|}-1\|_{L^1}
&\leq&\sum_{p\geq 1}\frac{1}{p!}(\|2u\|_{L^{2p}}\|v_n\|_{L^{2p}})^p\\
&\leq&\sum_{p\geq 1}\frac{1}{p!}\Gamma(p+1)(\|2u\|_{\mathcal L}\|v_n\|_{\mathcal L})^p=
1-\frac{1}{1-2\|u\|_{\mathcal L}\|v_n\|_{\mathcal L}},
\end{eqnarray*}
which implies that
\begin{equation}\label{3}
\lim_n\|{\rm e}^{2|uv_n|}-1\|_{L^1}= 0.
\end{equation}
This together with \eqref{1}, \eqref{2} and \eqref{3} end the proof of Proposition \ref{int1}.
\end{proof}
%%%%%%%%%%%%%%%%%%%%%%%%%%%%%%%%%%%%%%%%%%%%%%%%%%%%
The following elementary result is needed to derive an $L^\infty$ bound of the solution of the nonlinear heat equation in the defocusing case.
\begin{lem}\label{sq}
Let $(x_n)_{n\in\N}$ be a real valued sequence satisfying for some constants $C>1$ and $\beta>1$,
$$0\leq x_0\leq C_0^*:=C^{\frac{-1}{(\beta-1)^2}},\quad\mbox{and}\quad 0\leq x_{n+1}\leq C^nx_n^{\beta}.$$
Then
$$\lim_{n\rightarrow +\infty}x_n=0.$$
\end{lem}
\begin{proof}[Proof of Lemma \ref{sq}]
Let us define $y_n:=C^{\frac{(1+n(\beta-1))}{(\beta-1)^2}}x_n$. We have
$$0\leq y_n\leq y_n^{\beta},\quad\mbox{and}\quad y_0\leq 1.$$
Hence,  $y_n\leq 1$, and then  $0\leq x_n\leq C^{-\frac{(1+n(\beta-1))}{(\beta-1)^2}}$ which implies that $\displaystyle\lim_{n\rightarrow +\infty}x_n=0.$
\end{proof}
Finally, we recall the  following parabolic regularizing effect due to Brezis-Cazenave \cite{Br} that we will use to obtain a locally (in time) bounded solution to \eqref{eq2} with singular data. Consider the following linear heat equation with potential
\begin{equation}
\label{heatpotential}
\left\{
\begin{matrix}
\partial_t u-\Delta u - a(t,x) u=0\quad\mbox{in}\quad \Omega,\\
u=0\quad\mbox{on}\quad \partial\Omega,\\
u(0)= u_0,\\
\end{matrix}
\right.
\end{equation}
where $\Omega$ is a smooth bounded domain of $\R^2$.

\begin{thm}[ see {\bf Theorem A.1} in \cite{Br}]
\label{BrCaz}
Let $0<T<\infty$, $\sigma>1$, and let $a\in L^\infty([0,T]; L^\sigma)$. Given $u_0\in L^r$, $1\leq r<\infty$, there exists a unique solution $u\in {\mathcal C}([0,T]; L^r)\cap L^\infty_{loc}([0,T]; L^\infty)$ of equation \eqref{heatpotential}.
\end{thm}
%%%%%%%%%%%%%%%%%%%%%%%%%%%%%%%%%%%%%%%%%%%%%%%%%%%%%%%%%%%%%%%%%%%%%%%%%%%%%%%%%%%%%%%%%%%%%%%%%%%%%%%%%%%%%%%%%%%%%%%%%%%%%%%%%%%%%%%%%%%%%%%%%%%%%%%%%%%%%%%%%%%
\section{$H^1$-theory: proof of Theorem \ref{MainH1}}
%%%%%%%%%%%%%%%%%%%%%%%%%%%%%%%%%%%%%%%%%%%%%%%%%%%%%%%%%%%%%%%%%%%%%%%%%%%%%%%%%%%%%%%%%%%%%%%%%%%%%%%%%%%%%
This section is devoted to the proof of Theorem \ref{MainH1}. We divide the proof into several steps. First, we show the existence of a local solution to \eqref{eq2}, regardless of the sign of the nonlinearity. In the second step, we
prove the uniqueness in $C([0,T); H^1)$.  This result, is not straightforward, although it follows Brezis-Cazenave's steps. Then we show that in the defocusing case we can extend
the solution globally in time. Finally, we establish a finite time blow-up result in the focusing case.\\

\subsection{Local existence}

We summarize the result in the following Theorem.
\begin{thm}\label{t1}
Let $u_0\in H^1(\R^2)$. Then, there exist $T>0$ and a solution $u$ to \eqref{eq2} in the class
$$C([0,T);H^{1}({\mathbb{R}}^{2})).$$
\end{thm}
\begin{proof}[Proof of Theorem \ref{t1}]
The idea here is similar to the one used in \cite{Ib4,Ib3,Ib2}. Indeed, we decompose the
initial data to a regular part and a small one. We prove the existence of a local solution $v$ to \eqref{eq2}
associated to the regular initial data. Then to recover a solution of our original problem we solve a perturbed equation satisfied by $w:=u-v$ with small data.\\
%%%%%%%%%%%%%%%%%%%%%%%%%%%%%%%%%%%%%%%%%%%%%%%%%%%%%%%%%%%%%%%%%%%%%%%%%%%%%%%%%%%%%%%

%%%%%%%%%%%%%%%%%%%%%%%%%%%%%%%%%%%%%%%%%%%%%%%%%%%%%%%%%%%%%%%%%%%%%%%%%%%%%%%%
We start by showing the local existence in $H^1\cap L^\infty(\R^2)$ as claimed in the following proposition.
\begin{prop}\label{lm01}
Let $u_0\in (H^1\cap L^{\infty})(\R^2)$. Then, there exists $T>0$ (depending upon $u_0$) and a solution $u$ to \eqref{eq2} in the class
$${\mathcal C}([0,T);(H^{1}\cap L^{\infty})({\mathbb{R}}^{2})).$$
\end{prop}
%%%%%%%%%%%%%%%%%%%%%%%%%%%%%%%%%%%%%%%%%%%%%%%%%%%%%%%%%%%%%%%%%%%%%%%%%%%%%%%%%%%%%%%%%%%%%%%%%%%%%%%%%%%%%%%%
\begin{proof}[Proof of Proposition \ref{lm01}]
 Let the space
 $$\widetilde{X}_T:= \mathcal C([0,T);H^1\cap L^{\infty}(\R^2)),\quad T>0,$$
 endowed with the norm
 $$\|u\|^{\widetilde{}}_T:=\|u\|_{L_T^{\infty}(H^1(\R^2))}+\|u\|_{L_T^{\infty}(L^{\infty}(\R^2))}.$$
 Recall that $(X_T,\|\,.\,\|^{\widetilde{}}_T)$ is a Banach space. Set $v:={\rm e}^{t\Delta}u_0$ and define the map
 $$\Phi:u\longmapsto \int_0^t{\rm e}^{(t-s)\Delta}f((u+v)(s))ds.$$
Let $B_{T}(r)$ be the ball in $\widetilde{X}_{T}$ with center zero and radius $r>0$. We prove that for some $T,r>0$, the map $\Phi$ is a contraction from $B_{T}(r)$ into itself.\\
Applying the energy estimate \eqref{enrg} to ${u_1},{u_2}\in B_{T}(r)$ and the smoothing effect \eqref{smoth}, we obtain
\begin{eqnarray*}
\|\Phi(u_1)-\Phi(u_2)\|^{\widetilde{}}_T&\lesssim&\|f(u_1+v)-f(u_2+v)\|_{L^{1}_T(L^{2}({\mathbb{R}}^{2}))}+\|\nabla(f(u_1+v)-f(u_2+v))\|_{L^{1}_T(L^{2}({\mathbb{R}}^{2}))}\\
&+&\|f(u_1+v)-f(u_2+v)\|_{L^{1}_T(L^{\infty}({\mathbb{R}}^{2}))}\\
&\lesssim&\mathcal{A}+\mathcal{B}+\mathcal{C}.
\end{eqnarray*}
Let us start to estimate $\mathcal{A}$. Set $w:=u_1-u_2$ and $v_i:=u_i+v$, $i\in\{1,2\}$. Using Lemma \ref{f}, we infer
\begin{eqnarray*}
\|f(v_1)-f(v_2)\|_{L^2(\R^2)}
&\lesssim&\|w\sum_{i=1,2}\Big({\rm e}^{  2v_i^2}-1\Big)\|_{L^2(\R^2)}\\
&\lesssim&\|w\|_{L^2}\sum_{i=1,2}{\rm e}^{  2\|v_i\|_{L^{\infty}}^2}.\\
\end{eqnarray*}
Using the fact that $\|v_i\|_{L^{\infty}}\leq r+\|u_0\|_{L^{\infty}}$, it follows that
\begin{eqnarray*}
\|f(v_1)-f(v_2)\|_{L^2(\R^2)}\lesssim\|w\|_{L^2}{\rm e}^{2  (r+\|u_0\|_{L^{\infty}(\R^2)})^2}.
\end{eqnarray*}
Therefore
\begin{eqnarray*}
\mathcal{A}
&\lesssim&{\rm e}^{  2(r+\|u_0\|_{L^{\infty}(\R^2)})^2}T\|w\|_{L^{\infty}(0,T, L^2(\R^2))}\\
&\leq &C_{0,r}T\|w\|^{\widetilde{}}_{T}.
\end{eqnarray*}
Similarly, we have
\begin{eqnarray*}
\mathcal{C}\leq C_{0,r}T\|w\|^{\widetilde{}}_{T}.
\end{eqnarray*}
Now, let us estimate the second term $\mathcal{B}$. We have
\begin{eqnarray*}
\|\nabla(f(v_1)-f(v_2))\|_{L^2(\R^2)}
&=&\|\nabla v_1 (f^{'}(v_1)-f^{'}(v_2))+(\nabla v_1-\nabla v_2) f^{'}(v_2)\|_{L^2(\R^2)}\\
&\leq&\|\nabla v_1 (f^{'}(v_1)-f^{'}(v_2))\|_{L^2(\R^2)}+\|\nabla w f^{'}(v_2)\|_{L^2(\R^2)}\\
&\leq& \mathcal{B}_1+\mathcal{B}_2.
\end{eqnarray*}
Arguing as before, we obtain
\begin{eqnarray*}
\mathcal{B}_2\lesssim {\rm e}^{2(r+\|u_0\|_{L^\infty})^2}\|w\|_{H^1(\R^2)}.
\end{eqnarray*}
It remains to estimate $\mathcal{B}_1$. Using \eqref{f2}, we infer
\begin{eqnarray*}
\mathcal{B}_1&\lesssim&\sum_{i=1,2}\|\nabla v_1w({\rm e}^{4  v_i^2}-1)^{1/2}\|_{L^2(\R^2)} \\
&\lesssim&\|\nabla v_1\|_{L^2(\R^2)}\|w\|_{L^{\infty}}\sum_{i=1,2}{\rm e}^{2\|v_i\|^2_{L^\infty}}\\
&\lesssim&\|w\|_{L^{\infty}}\|\nabla v_1\|_{L^2(\R^2)}{\rm e}^{2(r+\|u_0\|_{L^\infty})^2}.
\end{eqnarray*}
Hence
\begin{eqnarray*}
\mathcal{B}\lesssim (1+r+\|u_0\|_{H^1}) {\rm e}^{2(r+\|u_0\|_{L^\infty})^2} T\|w\|^{\widetilde{}}_T\leq C_{0,r} T\|w\|^{\widetilde{}}_T,
\end{eqnarray*}
which implies that
\begin{equation}\label{4}
\|\Phi(u_1)-\Phi(u_2)\|^{\widetilde{}}_T\leq C_{0,r}T\|w\|^{\widetilde{}}_T.
\end{equation}
Now, let us estimate $\|\Phi(u_1)\|^{\widetilde{}}_T$. Using \eqref{nrg}, we deduce
$$\|\Phi(u_1)\|^{\widetilde{}}_T\leq C(\|f(v_1)\|_{L_T^1(H^1(\R^2))}+\|f(v_1)\|_{L_T^1(L^{\infty}(\R^2))}).$$
On the other hand, taking in the precedent computations $v_2=0$, we obtain
\begin{eqnarray*}
\|\Phi(u_1)\|^{\widetilde{}}_T
&\leq& C_{0,r}T\|v_1\|_T\\
&\leq& C_{0,r}T.
\end{eqnarray*}
It follows that for $r,T>0$ small enough, $\Phi$ is a contraction of a ball of $X_T$. Let $u$ to be the fixed point of $\Phi$. Then $u+v$ is a local solution to \eqref{eq2}. This concludes the proof of Proposition \ref{lm01}.
\end{proof}
\begin{rem}Note that
$$\|u\|^{\widetilde{}}_T\leq r+C\|u_0\|_{H^1\cap L^{\infty}}.$$
\end{rem}

\medskip

Now we solve the perturbed problem. We decompose the initial data as follows
$u_0=(I-S_N)u_0+S_N u_0$ where $S_N=\sum_{j\leq N-1}\;\triangle_j$, $(\triangle_j)$ being an inhomogeneous frequency localization, and $N$ is a large integer to be fixed later.
Recall that $\|(I-S_N)u_0\|_{H^1}\stackrel{N}{\longrightarrow} 0$ and $S_N u_0\in (H^1\cap L^{\infty})(\R^2)$.\\
By Proposition \ref{lm01}, there exist a time $T_N>0$ and a solution $v$ to the problem \eqref{eq2} with data $S_N u_0$. Now, we consider the perturbed problem satisfied by $w:=u-v$ and with data $(I-S_N)u_0$. Namely, let
\begin{equation}
\label{0eq1}
\left\{
\begin{matrix}
\partial_t w-\Delta w=-f(v)+f(v+w)\\
u(0)= (I-S_N)u_0.\\
\end{matrix}
\right.
\end{equation}
Using a standard fixed point argument, we shall prove that $\eqref{0eq1}$ has a local solution in the space $X_T:=\mathcal C([0,T);H^1(\R^2))$ for a suitable $T>0$ to be chosen.\\
We denote by
$\|u\|_T:=\|u\|_{L^{\infty}([0,T]; H^1(\R^2))}$ and we recall that $(X_T,\|\,.\,\|_T)$ is a Banach space. \\
Set $w_l:={\rm e}^{t\Delta}(I-S_N)u_0$
and consider the map
 $$\Psi:u\longmapsto \int_0^t{\rm e}^{(t-s)\Delta}(f(u+v+w_l)-f(v))(s))ds.$$
Let $B_{T}(r)$ be the ball in $X_{T}$ of radius $r>0$ and centered at the origin. We prove that for some $T,r>0$, the map $\Psi$ is a contraction from $B_{T}(r)$ into itself.\\
Applying the energy estimate \eqref{enrg} to ${u_1},{u_2}\in B_{T}(r)$ and using the smoothing effect \eqref{smoth}, we infer
\begin{eqnarray*}
\|\Psi(u_1)-\Psi(u_2)\|_{T}&\lesssim&\|f(u_1+v+w_l)-f(u_2+v+w_l)\|_{L^{1}_T(L^{2}({\mathbb{R}}^{2}))}\\
&+&{T}^{\frac{3}{2}}\|\nabla(f(u_1+v+w_l)-f(u_2+v+w_l))\|_{L^{\infty}_T(L^{1}({\mathbb{R}}^{2}))}.
\end{eqnarray*}
Set $w:=u_1-u_2$ and $v_i:=u_i+v+w_l$. Using Lemma \ref{f}, we obtain
\begin{eqnarray*}
\|f(v_1)-f(v_2)\|_{L^2(\R^2)}
&\lesssim&\sum_{i=1,2}\|w({\rm e}^{2 v_i^2}-1)\|_{L^2(\R^2)}
\end{eqnarray*}
Since $|v_i|^2\leq 2(w_l+u_i)^2+2v^2$ and using the simple observation
\begin{eqnarray}
{\rm e}^{a+b}-1=({\rm e}^a-1)({\rm e}^b-1)+({\rm e}^a-1)+({\rm e}^b-1)
\end{eqnarray}
we have,
\begin{eqnarray*}
\|w({\rm e}^{2  v_i^2}-1)\|_{L^2}\leq \|w({\rm e}^{4  v^2}-1)\|_{L^2}&+&\|w({\rm e}^{4  (u_i+w_l)^2}-1)\|_{L^2}\\
&+&\|w({\rm e}^{4  v^2}-1)({\rm e}^{2(1+\epsilon)(u_i+w_l)^2}-1)\|_{L^2}.
\end{eqnarray*}
By H\"older inequality and Sobolev embedding, we have
$$
\|w({\rm e}^{4  v^2}-1)\|_{L^2}\leq \|w\|_{L^2}{\rm e}^{4\|v\|^2_{L^\infty}}\lesssim {\rm e}^{4\|v\|^2_{L^\infty}} \|w\|_{H^1},
$$
and
\begin{eqnarray*}
\|w({\rm e}^{4  (u_i+w_l)^2}-1)\|_{L^2}&\leq &\|w\|_{L^6}\|{\rm e}^{4(u_i+w_l)^2}-1\|_{L^3}\\
&\lesssim& \|w\|_{H^1}\|{\rm e}^{4(u_i+w_l)^2}-1\|_{L^3}.
\end{eqnarray*}
Denoting $\varepsilon_n:=\|(I-S_N)u_0\|_{H^1}$, we have $\|\nabla(u_i+w_l)\|_{L^2}\lesssim r+\varepsilon_n\stackrel{r,n}{\longrightarrow}0$. Hence, for $\alpha>0$, $p\geq 1$ and thanks to Moser-Trudinger  inequality we derive
\begin{eqnarray}\label{mtr}
\|{\rm e}^{\alpha (u_i+w_l)^2}-1\|_{L^p}&\leq& \|{\rm e}^{\alpha p(u_i+w_l)^2}-1\|_{L^1}^{\frac{1}{p}}\\
&\lesssim& (r+\varepsilon_n)^{2/p},\nonumber
\end{eqnarray}
and
$$
\|w({\rm e}^{4  (u_i+w_l)^2}-1)\|_{L^2}\lesssim \|w\|_{H^1}(r+\varepsilon_n)^{2/3}.
$$
Consequently,
$$
\|w({\rm e}^{4  v^2}-1)({\rm e}^{2(1+\epsilon)(u_i+w_l)^2}-1)\|_{L^2}\lesssim {\rm e}^{4\|v\|^2_{L^\infty}}\|w\|_{H^1}(r+\varepsilon_n)^{2/3}.
$$
Therefore,
\begin{eqnarray*}
\|f(v_1)-f(v_2)\|_{L_T^{1}(L^2)}
\leq C_{0,r}T\|w\|_{L_T^{\infty}H^1(\R^2)},
\end{eqnarray*}
where the constant $C_{0,r}$ depends only on $u_0$ and $r$. It remains to control $\|\nabla(f(v_1)-f(v_2))\|_{L_T^{\infty}(L^1(\R^2))}.$ We have
\begin{eqnarray*}
\|\nabla(f(v_1)-f(v_2))\|_{L^1(\R^2)}
&=&\|\nabla v_1 (f^{'}(v_1)-f^{'}(v_2))+(\nabla v_1-\nabla v_2) f^{'}(v_2)\|_{L^1(\R^2)}\\
&\leq&\|\nabla v_1 (f^{'}(v_1)-f^{'}(v_2))\|_{L^1(\R^2)}+\|\nabla w f^{'}(v_2)\|_{L^1(\R^2)}\\
&\leq& {\mathbf E}+{\mathbf F}.
\end{eqnarray*}
Arguing as before, we have
$${\mathbf E}\lesssim \|v\|_{H^1}{\rm e}^{4\|v\|^2_{L^\infty}}(1+(r+\varepsilon_n)^{1/2})\|w\|_{H^1},$$ and $${\mathbf F}\lesssim (\|v\|_{H^1}+2r+2\varepsilon_n){\rm e}^{4\|v\|^2_{L^\infty}}\|w\|_{H^1}.$$ Therefore
\begin{eqnarray*}
\|\nabla(f(v_1)-f(v_2))\|_{L_T^{\infty}(L^1(\R^2))}
&\leq&C_{0,r}\|w\|_T,
\end{eqnarray*}
which implies that
\begin{equation}%\label{4}
\|\Psi(u_1)-\Psi(u_2)\|_{T}\leq C_{0,r}(1+T)T^{\frac{1}{2}}\|w\|_T.
\end{equation}
Now we estimate $\|\Psi(u_1)\|_{T}$. Using \eqref{nrg}, we obtain
$$\|\Psi(u_1)\|_{T}\leq C\|f(v_1)-f(v)\|_{L_T^1(L^2(\R^2))}+ T^{\frac{3}{2}}\|\nabla (f(v_1)-f(v))\|_{L_T^{\infty}(L^1(\R^2))},$$
where we set $v_1:=u_1+v+w_l$. Taking $v_2=v$, in the precedent computations, we have
\begin{eqnarray*}
\|\Psi(u_1)\|_{T}
&\leq& C_{0,r}(1+T)T^{\frac{1}{2}}\|u_1+w_l\|_T\\
&\leq& C_{0,r}(r+\|u_0\|_{H^1(\R^2)})(1+T^{\frac{1}{2}})T^{\frac{1}{2}}
\end{eqnarray*}
In conclusion, for $T$ small enough, $\Psi$ is a contraction of some ball of $X_T$. We obtain the desired solution by taking $u+w_l$ where $u$ is the fixed point of $\Psi$. The proof is achieved.
\end{proof}
%%%%%%%%%%%%%%%%%%%%%%%%%%%%%%%%%%%%%%%%%%%%%%%%%%%%%%%%%%%%%%%%%%%%%%%%%%%%%%%%%%%%%%%%%%%%%%%%%%%%%%%%%%%%%%%%%%%%%
\subsection{Uniqueness in ${\mathcal C}([0,T[; H^1(\R^2))$}
%%%%%%%%%%%%%%%%%%%%%%%%%%%%%%%%%%%%%%%%%%%%%%%%%%%%%%%%%%%%%%%%%%%%%%%%%%%%%%%%%%%%%%%%%%%%%%%%%%%%%%%%%%%%
This subsection is devoted to the proof of the uniqueness part of Theorem \ref{MainH1}. More precisely, we prove an unconditional uniqueness result.
\begin{thm}\label{Uniq}
The solution given in Theorem \ref{t1} is unique in the class
$${\mathcal C}([0,T);H^{1}({\mathbb{R}}^{2})).$$
\end{thm}
\begin{proof}[Proof of Theorem \ref{Uniq}]
Let $u$, $v\in \mathcal C([0,T];H^1)$ be two solutions to \eqref{eq2} with same data $u_0$ and set $w:=u-v$. Define the potential
$$
a(t,x):=\left\{\begin{array}{rl}
&\frac{f(u)-f(v)}{w},\quad \mbox{if }w\neq 0\\
&f^{'}(u),\quad\quad \mbox{if } w=0 ,
\end{array}\right.
$$
so that,
$$
w(t)=\int_0^t {\rm e}^{(t-s)\Delta}a(s)w(s)\;ds.
$$The following Lemma can be seen as an extension of Brezis-Casenave's result \cite{Br} to the two dimensional case. The crucial point is to show the continuity of the potential term (continuity at $t=0$). As pointed out in \cite{Br}, this result seems to be open if the potential is only $L^{\infty}$ in time.
\begin{lem}\label{un}
Let $a\in \mathcal C([0,T];L^p(\R^2))$ and $u\in L^{\infty}((0,T); L^q(\R^2))$ with $2\leq q<\infty,1< p< \infty$, $\frac{1}{p}+\frac{1}{q}\neq 1$ and such that
$$
u(t)=\int_0^t{\rm e}^{(t-s)\Delta}a(s)u(s)ds,\quad\forall\;t\in [0,T].
$$
Then $u=0$ on $[0,T]$.
\end{lem}
\begin{proof}
It is clear that $au\in L^\infty([0,T), L^r(\R^2))$ where
$\displaystyle{\frac{1}{r}=\frac{1}{q}+\frac{1}{p}}$ ($1< r<\infty$), so that by maximal regularity $u\in L^{\widetilde{p}}((0,T),W^{1,r}(\R^2))$ for all ${\widetilde{p}}<\infty$ and satisfies for almost every $t\in (0,T)$ the next equation in $L^r(\R^2)$ ,
\begin{equation}
\label{u}
\partial_tu-\triangle u=a u\,.
\end{equation}
Let $t_0\in[0,T]$, $\psi\in D(\R^2)$ and $a_n:=\min\{n,\max\{a,-n\}\}$. Denote by $v_n$ the solution to
$$\left\{\begin{array}{cccc}
 -\partial_tv_n-\Delta v_n&=&a_nv_n\quad \mbox{in }(0,t_0)\times\R^2,\\
v_n(t_0)=\psi.
\end{array}\right.$$
Multiplying \eqref{u} by $v_n$ and then integrating on $(0,t_0)\times\R^2$, we have
$$
\int_0^{t_0}\int_{\R^2}(\partial_t u v_n-\Delta u v_n)\;dx\;dt=\int_0^{t_0}\int_{\R^2}a u v_n\;dx\;dt
$$
Hence
\begin{eqnarray}\label{hol}
\int_{\R^2}u(t_0)\psi\;dx
&=&\int_0^{t_0}\partial_t (u v_n)\;dx\;dt\nonumber\\
&=&\int_0^{t_0}\int_{\R^2} u  \partial_t{v_n}+(\Delta u+au) v_n\;dx\;dt\nonumber\\
&=&\int_0^{t_0}\int_{\R^2} (a-a_n)uv_n\;dx\;dt.
\end{eqnarray}
In order to prove that $u=0$ on $[0,T]$ it is sufficient to show that
\begin{equation}\label{a}
a_n\stackrel{n\rightarrow\infty}{\longrightarrow}a \mbox{ in } \mathcal C([0,T]; L^{p}(\R^2))
\end{equation}
 and
 \begin{equation}\label{v}
\sup_{n\geq 0}\|v_n\|_{L^\infty([0,t_0),L^{r'}(\R^2))}\leq C_{r'}\|\psi\|_{L^{r'}(\R^2)},
\end{equation}
where $\displaystyle{\frac{1}{r'}=1-\frac{1}{r}}$.
 First let us prove \eqref{a}. By contradiction, there exist $\eta_0>0$ such that $\displaystyle\sup_{0\leq t\leq T}\|a_n(t)-a(t)\|_{L^{p}}>\eta_0$ uniformly on $n\in\N$. So, for fixed $n\in\N$, there exist $s\in[0,T]$ and a real sequence $t_k \in [0,T]$, $t_k\stackrel {k\rightarrow\infty}{\longrightarrow} s$ such that
\begin{equation}\label{e}
\|a_n(t_k)-a(t_k)\|_{L^{p}}\stackrel{k\rightarrow\infty}{\longrightarrow}\sup_{0\leq t\leq T}\|a_n(t)-a(t)\|_{L^{p}}.
\end{equation}
This implies that, for any $k\in\N$,
\begin{equation}\label{e1}
\|a_n(t_k)-a(t_k)\|_{L^{p}}\geq \eta_0\quad\forall k\in\N.
\end{equation}
Since $a_n(t,x)\stackrel{n\rightarrow\infty}{\longrightarrow}a(t,x)$ almost everywhere, $|a_n|\leq|a|$ and $a(t)\in L^{p}(\R^2)$, by
Lebesgue Theorem , we have
\begin{equation}\label{e1'}
\|a_n(t)-a(t)\|_{L^{p}}\stackrel{n\rightarrow\infty}{\longrightarrow} 0\quad\forall t\in[0,T].
\end{equation}
Writing
$$
\|a_n(t_k)-a(t_k)\|_{L^{p}}\leq \|a_n(t_k)-a(s)\|_{L^{p}}+\|a(s)-a(t_k)\|_{L^{p}},
$$
and using the fact that $a\in \mathcal C([0,T]; L^{p}(\R^2))$, we deduce
$$\eta_0\leq
\limsup_{k\longrightarrow \infty}\|a_n(t_k)-a(t_k)\|_{L^{p}}\leq\limsup_{k\longrightarrow \infty}\|a_n(t_k)-a(s)\|_{L^{p}}=\|a_n(s)-a(s)\|_{L^{p}}.
$$

This obviously contradicts \eqref{e1}-\eqref{e1'}.\\
%%%%%%%%%%%%%%%%%%%%%%%%%%%%%%%%%%%%%%%%%%%%%%%%%%%%%%%%%%%%%%%%%%%%%%%%%%%%%%%%%%%%%%%%%%%%%
Now, we prove \eqref{v}. We take $\widetilde{v}_n(t):=v_n(t_0-t)$, and $b_n(t)=:a_n(t_0-t)$, we have
$$\left\{\begin{array}{cccc}
 \partial_t\widetilde v_n-\Delta \widetilde{v}_n&=&b_n\widetilde{v}_n ,\\
\widetilde{v}_n(0)=\psi.
\end{array}\right.$$
First, we multiply the precedent equation by $|\widetilde{v}_n|^{r'-2}\widetilde{v}_n$ then we integrate over $\R^2$, we obtain
\begin{eqnarray*}
\frac{1}{r'}\frac{d}{dt}\int_{\R^2}|\widetilde{v}_n(t,x)|^{r'}\;dx+\frac{4(r'-1)}{r'^2}\int_{\R^2}|\nabla|\widetilde{v}_n|^{r'/2}|^2\;dx
&\leq&\int_{\R^2}|b_n||\widetilde{v}_n|^{r'}\;dx\\
&\leq&\int_{\R^2}|b||\widetilde{v}_n|^{r'}\;dx.
\end{eqnarray*}
In the last inequality we used $|b_n|\leq |b|$ because $|a_n|\leq|a|$,   where $b=a(t_0-.)$ on $[0,t_0]$.\\
Using the fact that $|b_j|\leq j$ and  Sobolev embedding, we get
\begin{eqnarray*}
\int_{\R^2}|b||\widetilde{v}_n|^{r'}\;dx&\leq&\int_{\R^2}|b-b_j||\widetilde{v}_n|^{r'}\;dx+\int_{\R^2}|b_j||\widetilde{v}_n|^{r'}\;dx\\
&\leq&\|b-b_j\|_{L^{p}}\|{\widetilde{v}_n}^{r'/2}\|_{L^{2p^{'}}}^2+j\int_{\R^2}|\widetilde{v}_n|^{r'}\;dx,\\
&\leq&C\|b-b_j\|_{L^{p}}\|{\nabla|\widetilde{v}_n}|^{r'/2}\|_{L^{2(1+\frac{1}{\varepsilon})}}^2+(j+C)\|\widetilde{v}_n(t)\|_{L^{r'}}^{r'}
\end{eqnarray*}
Since $b_j\stackrel{j\rightarrow\infty}{\longrightarrow}b \mbox{ in } \mathcal C([0,T]; L^{p}(\R^2))$, we choose $j\geq 0$ large enough such that
$$C\displaystyle{\|b-b_j\|_{L^{p}}\leq \frac{4(r'-1)}{{r'}^2}}\,.$$
Therefore
$$
\frac{1}{r'}\frac{d}{dt}\|\widetilde{v}_n(t)\|_{L^{r'}}^{r'}\leq (j+C)\|\widetilde{v}_n(t)\|_{L^{r'}}^{r'}.
$$
Using Gronwall Lemma, il follows that
$$\|\widetilde{v}_n(t)\|_{L^{r'}}^{r'}\leq \|\psi\|_{L^{r'}}^{r'}{\rm e}^{(j+C)r't}$$
which conclude the proof of \eqref{v}.\\
The proof of the Lemma \ref{un} is achieved.
\end{proof}
Now we prove that $a\in \mathcal C([0,T];L^2)$. We proceed by contradiction.
Assume that there exists $\varepsilon>0$, $t\in [0,T]$ and a sequence of real numbers $(t_n)$ in $[0,T]$ such that
\begin{equation}\label{abs}
t_n\rightarrow t\quad\mbox{and}\quad\|a(t_n)-a(t)\|_{L^{2}}>\varepsilon,\quad \forall n\in\N.
\end{equation}
 Denote $u_n:=u(t_n),v_n:=v(t_n)$ and $w_n:=w(t_n)$. Recall that
$u,v\in \mathcal C([0,T];H^1)$. So up to extraction of a subsequence, we have
$$a(t_n)\rightarrow a(t)\quad\mbox{almost everywhere}.$$
 Moreover, by a convexity argument\begin{eqnarray*}
|a(t_n)|&\leq&{\rm e}^{2u_n^2}-1+{\rm e}^{2v_n^2}-1.
\end{eqnarray*}
Since $u\in \mathcal C([0,T];H^1)$, using Proposition \ref{int1}, we infer
$${\rm e}^{2u_n^2}-1\rightarrow {\rm e}^{2u^2}-1\quad\mbox{and}\quad {\rm e}^{2v_n^2}-1\rightarrow {\rm e}^{2v^2}-1\quad\mbox{in}\quad L^{2}(\R^2).$$
Thus, there exists $\phi\in L^{2}$ such that
$$|a(t_n)|\leq \phi.$$
Using Lebesgue theorem, we deduce
$$  a(t_n)\rightarrow a(t)\quad\mbox{in}\quad L^{2}.$$
This contradicts \eqref{abs}, and we conclude that $a\in \mathcal C([0,T];L^{2}).$\\

\noindent{\bf End of the proof of Theorem \ref{Uniq}}.\\
It is sufficient to check assumptions of Lemma \ref{un}. Obviously $w\in L^{\infty}([0,T],L^q(\R^2))$ for every $2\leq q<\infty$.
\end{proof}

%%%%%%%%%%%%%%%%%%%%%%%%%%%%%%%%%%%%%%%%%%%%%%%%%%%%%%%%%%%%%%%%%%%%%%%%%%%%%%%%%%%%%%%%%%%%%%%%%%%%%%%%%%%%%%%%%%%%%%%%%%%%%%%%%%%%%%%%%%%%%%%%%%%%
%%%%%%%%%%%%%%%%%%%%%%%%%%%%%%%%%%%%%%%%%%%%%%%%%%%%%%%%%%%%%%%%%%%%%%%%%%%%%%%%%%%%%%%%%%%%%%%%%%%%%%%%%%%%%%%%%%%%%%%%%%%%%%%%%%%%%%%%%%%%%%%
\subsection{Global existence}
%%%%%%%%%%%%%%%%%%%%%%%%%%%%%%%%%%%%%%%%%%%%%%%%%%%%%%%%%%%%%%%%%%%%%%%%%%%%%%%%%%%%%%%%%%%%%%%%%%%%%%%%%%%%%%%%%%%%%%%%%%%%%%%%
In this subsection we prove a global well-posedness result in the defocusing case.
\begin{thm}\label{Global}
Let $u_0\in H^1(\R^2)$ and assume that $f(u)=-u({\rm e}^{  u^2}-1)$. Then, there exists a unique global solution to \eqref{eq2} in the class
$$\mathcal{C}(\R,H^{1}(\R^{2})).$$
\end{thm}
In the defocusing case, we prove that for $u_0\in H^1(\R^2)$, the solution $u\in \mathcal C([0,T);H^1(\R^2))$ to \eqref{eq2} satisfies $ u\in L^{\infty}_{loc}(]0,T[; L^{\infty}(\R^2))$ then we conclude using the next standard blow-up criterion (see for example \cite{Br}).
\begin{lem}\label{lm1}
Let $u_0\in H^1(\R^2)$ and $u\in \mathcal C([0,T^*);H^1(\R^2))$ solution to \eqref{eq2}. Assume that $T^*<\infty$, then
$$\limsup_{t\rightarrow T^*} \|u(t)\|_{L^{\infty}(\R^2)}=+\infty.$$
\end{lem}
The next Proposition proves that the solution is bounded. The proof is in the spirit of Caffarelli-Vasseur \cite{cv} and is based on the energy estimate.
\begin{prop}\label{lm0}
Assume that $u_0\in H^1(\R^2)$ and $u\in \mathcal C([0,T);H^1(\R^2))$ solution to \eqref{eq2} with $f(u)=-u({\rm e}^{ u^2}-1)$. Then
$$u\in L^{\infty}_{loc}(]0,T[,L^{\infty}(\R^2))\quad\mbox{and}\quad \|u(t)\|_{L^{\infty}}\leq \sqrt{2}\,\|u_0\|_{L^2},\quad\forall\;0<t<T.$$
\end{prop}
\begin{rem}
Actually we will obtain a more precise estimate, namely
$$
\|u(t)\|_{L^{\infty}(\R^2)}\leq 2^{\frac{\alpha^2+10\alpha-12}{2\alpha (\alpha-2)}}\,t^{-\frac{1}{\alpha}}\,\|u_0\|_{L^2},\quad\forall\;\alpha>2.
$$
\end{rem}
\begin{proof}[Proof of Proposition \ref{lm0}].
%%%%%%%%%%%%%%%%%%%%%%%%%%%%%%%%%%%%%%%%%%%%%%%%%%%%%%%%%%%%%%%%%%%%%%%%%%%%%%%%%%%%%%%%%%%%%%%%%%%%%%%%%%%%%%%%%%%%%%%%%%%%%%%%%%%%%%%%%%
Let $M>0$ to fix later and
$$c_k:=M(1-2^{-k}),\quad u_k:=(u-c_k)_+=(u-c_k)\chi_{\{u>c_k\}}.$$
Since $uf(u)\leq 0$, we have the following estimate
\begin{equation}\label{e2}
\frac{d}{dt}\int_{\R^2}u_k^2(t)dx+\int_{\R^2}|\nabla u_k(t)|^2dx\leq 0.
\end{equation}
Let $t_0>0$, $T_k:=t_0(1-2^{-k})$, and
$$U_k:=\sup_{t\geq T_k}\Big(\int_{\R^2}u_k^2(t,x)dx\Big)+2\int_{T_k}^{+\infty}\int_{\R^2}|\nabla u_k(t,x)|^2dxdt.$$
Let $T_{k-1}\leq s\leq T_k\leq t\leq t_0$. Integring \eqref{e2} between $s,t$ and $s,\infty$ we obtain
$$U_k\leq 2\|u_k(s)\|_{L^2(\R^2)}^2.$$
Thus integring between $T_{k-1}$ and $T_k$, we have
\begin{equation}\label{e0}
U_k\leq\frac{2^{k+1}}{t_0}\|u_k\|_{L^2((T_{k-1},T_k)\times\R^2)}^2.
\end{equation}
Moreover, using the interpolation estimate
\begin{equation}
\|u\|_{L^r}\leq\|u\|_{L^2}^{\frac{2}{r}}\|\nabla u\|_{L^r}^{1-\frac{2}{r}},\quad 2\leq r<\infty,
\end{equation}
we obtain
$$\|u_k(t)\|_{L^r}\leq\|u_k(t)\|_{L^2}^{\frac{2}{r}}\|\nabla u_k(t)\|_{L^r}^{1-\frac{2}{r}}\leq\|u_k(0)\|_{L^2}^{\frac{2}{r}}\|\nabla u_k(t)\|_{L^r}^{1-\frac{2}{r}}\,.$$
Therefore
\begin{eqnarray*}
\|\nabla u_k\|_{L^2((T_k,\infty)\times\R^2)}^2
&\geq&\Big( \frac{\|u_k\|_{L^r((T_k,\infty)\times\R^2)}}{\|u_k(0)\|_{L^2}^{\frac{2}{r}}}\Big)^{\frac{2r}{r-2}}\\
&\geq&\Big( \frac{1}{2\|u_0\|_{L^2}}\Big)^{\frac{4}{r-2}}\|u_k\|_{L^r((T_k,\infty)\times\R^2)}^{\frac{2r}{r-2}}\\
&\geq&C_0\|u_k\|_{L^r((T_k,\infty)\times\R^2)}^{\frac{2r}{r-2}}.
\end{eqnarray*}
Thus
\begin{equation}\label{e3}
U_k\geq 2C_0\|u_k\|_{L^r((T_k,\infty)\times\R^2)}^{\frac{2r}{r-2}}.
\end{equation}
On the other hand, if $u_k>0$ then $u_{k-1}>2^{-k}M$, so
$$\chi_{\{u_k>0\}}\leq\Big(\frac{2^k}{M}u_{k-1}\Big)^{\alpha},\quad\forall \alpha>0.$$
Now, by \eqref{e0}, and for given $\alpha>0$,
\begin{eqnarray*}
U_k
&\leq&\frac{2^{k+1}}{t_0}\int_{(T_{k-1},\infty)\times\R^2}u_{k-1}^2\chi_{\{u_k>0\}}dxdt\\
&\leq&\frac{2^{1+k(\alpha+1)}}{M^{\alpha}t_0}\int_{(T_{k-1},\infty)\times\R^2}u_{k-1}^{2+\alpha}dxdt\\
&\leq&\frac{2^{1+k(\alpha+1)}}{M^{\alpha}t_0}\|u_{k-1}\|_{L^{2+\alpha}({(T_{k-1},\infty)\times\R^2})}^{2+\alpha}.
\end{eqnarray*}
Using \eqref{e3}, we infer
\begin{eqnarray*}
U_k&\leq&A\, C^{k-1}\,U_{k-1}^{\frac{\alpha}{2}},
\end{eqnarray*}
where
$$
A:=\frac{2^{\alpha/2+4}\,\|u_0\|_{L^2}^2}{M^\alpha\, t_0},\quad\mbox{and}\quad C:=2^{\alpha+1}\,.
$$
To conclude the proof, we shall use Lemma \ref{sq}. Indeed, taking $\alpha>2$, $M:=2^{\frac{\alpha^2+10\alpha-12}{2\alpha (\alpha-2)}}\,t_0^{-\frac{1}{\alpha}}\,\|u_0\|_{L^2}$, and applying Lemma \ref{sq} to the sequence $x_k:=A^{\frac{2}{\alpha-2}}\,U_k$, we obtain
$$\lim_{k\rightarrow +\infty}U_k=0.$$
Thus $u\leq M$ for $t\geq t_0$.The same proof on $-u$ gives the same bound for $|u|$. We obtain finally
$$\|u(t_0)\|_{L^{\infty}(\R^2)}\leq 2^{\frac{\alpha^2+10\alpha-12}{2\alpha (\alpha-2)}}\,t_0^{-\frac{1}{\alpha}}\,\|u_0\|_{L^2},\quad\forall\;\alpha>2.$$
Letting $\alpha$ to infinity, we conclude the proof of Proposition \ref{lm0}.
\end{proof}
%%%%%%%%%%%%%%%%%%%%%%%%%%%%%%%%%%%%%%%%%%%%%%%%%%%%%%%%%%%%%%%%%%%%%%%%%%%%%%%%%%%%%%%%%%%%%%%%%%%%%%%%%%%%%

\subsection{Blowing-up solutions}
%%%%%%%%%%%%%%%%%%%%%%%%%%%%%%%%%%%%%%%%%%%%%%%%%%%%%%%%%%%%%%%%%%%%%%%%%%%%%%%%%%%%%%%%%%%%%%%%%%%%%%%%%%%%%
Recall the energy
$$
J(t):=J(u(t))=\frac{1}{2}\|\nabla u(t)\|_{L^2}^2-\int_{\R^2}F(u(t))\;dx,
$$
with $F(u)=\frac{1}{2}\left({\rm e}^{  u^2}-1-u^2\right)$. We show that all solutions with non-positive energy have a finite lifespan time. More precisely

\begin{prop}\label{Blowup}
Let $u_0\in H^1(\R^2)\backslash\{0\}$ such that $J(u_0)\leq 0$ and $u\in\mathcal C([0,T^*[; H^1(\R^2))$ be the maximal solution to \eqref{eq2} with data $u_0$. Then $T^*<\infty$.
\end{prop}
The proof is standard and follows for example \cite{L} (see also \cite{IMN} in the context of the Klein-Gordon equation). It consists in following the evolution in time of the function
$$
y(t):=\frac{1}{2}\,\int_0^t\|u(s)\|_{L^2}^2\;ds.
$$
\begin{proof}[Proof of Proposition \ref{Blowup}]

First, observe that since we have removed the quadratic term from the nonlinearity, then $f(u)$ enjoys the following  property for a certain positive number $\varepsilon$
\begin{eqnarray}
\label{superquadratic}
\Big(u f(u)-2F(u)\Big)\geq \varepsilon F(u)\,.
\end{eqnarray}
Next, multiplying \eqref{eq2} by $u$, integrating in space we obtain
$$
J'(t)=-\|\partial_tu(t)\|_{L^2}^2,
$$
and by an integration in time
\begin{eqnarray}
\label{evolution}
J(t)=J(0)-\int_0^t\int_{\R^2}(\partial_tu)^2(s,x)\;dx\,ds.
\end{eqnarray}
Finally, a straight calculation shows that
\begin{eqnarray}
\label{secondderivative}
y''(t)&=&-\|\nabla u\|_{L^2}^2+\int_{\R^2} u f(u)\;dx\nonumber\\
&\geq& \frac{2+\varepsilon}{2}\Big(\int_{\R^2}2 F(u)\;dx-\|\nabla u\|_{L^2}^2\Big)\nonumber\\
&\geq& (2+\varepsilon)\Big(\int_0^t\int_{\R^2}\partial_t u^2\;dx\,ds-J(0)\Big),
\end{eqnarray}
where we used property \eqref{superquadratic} in the second estimate and identity \eqref{evolution} in the last one.
Now, the proof goes by contradiction assuming that $T^*=\infty$. We have\\
{\sf Claim 1:} There exists $t_1>0$ such that $\int_0^{t_1}\|\partial_t u(s)\|_{L^2}^2\;ds>0$.\\
Indeed, otherwise $u(t)=u_0$ almost everywhere and thus $u$ solves the elliptic stationary equation $\Delta u=-f(u)$. Then $\|\nabla u\|_{L^2(\R^2)}^2= \int_{\R^2}uf(u)dx$, and therefore
$$
0\leq\varepsilon \int_{\R^2} F(u_0)\;dx\leq \int_{\R^2}\Big(u_0f(u_0)-2F(u_0)\Big)\;dx=2J(0)\leq 0
$$
giving  $u_0=0$ which is an absurdity.\\
{\sf Claim 2:} For any $0<\alpha<1$, there exists $t_\alpha>0$ such that
$$
(y'(t)-y'(0))^2\geq \alpha y'(t)^2, \quad t\geq t_\alpha.
$$
The claim immediately follows from the first one observing that
$$
\lim_{t\rightarrow\infty}y(t)=\lim_{t\rightarrow\infty}y'(t)=+\infty.
$$
{\sf Claim 3}: One can choose $\alpha=\alpha(\varepsilon)$ such that
\begin{eqnarray}
\label{convexity}
y(t)y''(t)\geq(1+\alpha)y'(t)^2,\quad t\geq t_\alpha.
\end{eqnarray}
Indeed, we have
\begin{eqnarray}
\nonumber
y(t)y''(t)&\geq&\frac{2+\varepsilon}2\Big(\int_0^t\int_{\R^2}u^2\;dxds\Big)\,\Big(\int_0^t\int_{\R^2}\partial_tu^2\;dxds\Big)\\
\nonumber
&\geq&\frac{2+\varepsilon}2\Big(\int_0^t\int_{\R^2}u\partial_tu \;dxds\Big)^2\\
\nonumber
&\geq&\frac{2+\varepsilon}2(y'(t)-y'(0))^2\\
\nonumber
&\geq&\frac{(2+\varepsilon)\alpha}2(y'(t))^2,
\end{eqnarray}
where we used \eqref{secondderivative} in the first estimate, Cauchy-Schwarz inequality in the second and {\sf Claim 2} in the last one.
Now choose $\alpha$ such that $\frac{(2+\varepsilon)\alpha}{2}>1$  then
$$
y(t)y''(t)\geq\frac{(2+\varepsilon)\alpha}{2} y(t)^2.
$$
The fact that this ordinary differential inequality blows up in finite time contradicts our assumption that the solution was global.
\end{proof}
%%%%%%%%%%%%%%%%%%%%%%%%%%%%%%%%%%%%%%%%%%%%%%%%%%%%%%%%%%%%%%%%%%%%%%%%%%%%%%%%%%%%%%%%%%%%%%%%%%%%%%%%%%%%%%%%%%%%%%%%%%%%%%%%%%%%%%%%%%%%%%%%%%%%%%%%%%%%5

%%%%%%%%%%%%%%%%%%%%%%%%%%%%%%%%%%%%%%%%%%%%%%%%%%%%%%%%%%%%%%%%%%%%%%%%%%%%%%%%%%%%%%%%%%%%%%%%%%%%%%%%%%%%%%%%%%%%%%%%%%%%%%%%%%%%%%%%%%%%%%%%%%%%%%%%%%%%%%%%%%%%%%%%%%%%%%%%%%%%%
\section{The Cauchy problem in Orlicz space: proof of Theorem \ref{MainOrlicz}}
%%%%%%%%%%%%%%%%%%%%%%%%%%%%%%%%%%%%%%%%%%%%%%%%%%%%%%%%%%%%%%%%%%%%%%%%%%%%%%%%%%%%%%%%%%%%%%%%%%%%%%%%%%%%%%%%%%%%%%%%%%%%%%%%%%%%%%%%%%%%%%%%%%%%%%%%%%%%%%%%%%%%%%%%%%5
We start this section with the following definition of singular solution.
\begin{defi}
Recall that $B_1$ is the unit ball of $\R^2$. By a singular solution of
$$\left\{\begin{array}{rl}
-\Delta u=f(u) \quad\mbox{in}\quad B_1\\
u=0 \quad \mbox{on}\quad \partial B_1\\
u>0
\end{array}\right.$$
we mean a function $u\in {\mathcal C}^2(B_1\backslash\{0\})$ satisfying
$$
\limsup_{x\to 0} u(x)=\infty\,.
$$
\end{defi}
First we construct a stationary singular solution $Q$ to \eqref{eq2}. Second, we prove the existence of local solution to the Cauchy problem \eqref{eq2} in the Orlicz space $\mathcal L$ with data $Q$. Third we prove a regularizing effect of the heat equation. The nonuniqueness result given by Theorem \ref{MainOrlicz} immediately follows.
%%%%%%%%%%%%%%%%%%%%%%%%%%%%%%%%%%%%%%%%%%%%%%%%%%%%%%%%%%%%%%%%%%%%%%%%%%%%%%%%%%%%%%%%%%%%%%%%%%%%%%%%%%%%%%%%%%%%%%%%%%%%%%%%%%%%%%%%%%%%%%%%%%%%%%%%%%%%5
\subsection{Singular solutions}
%%%%%%%%%%%%%%%%%%%%%%%%%%%%%%%%%%%%%%%%%%%%%%%%%%%%%%%%%%%%%%%%%%%%%%%%%%%%%%%%%%%%%%%%%%%%%%%%%%%%%%%%%%%%%%%%%%%%%%%%%%%%%%%%%%%%%%%%%%%%%%%%%%%%%%%%%%%%%%%%%%%%%%%%%%%%%%%%%%%%%
The main goal of this subsection is to prove the next result.
\begin{thm}
\label{Singular}
The following singular elliptic problem
$$\left\{\begin{array}{rl}
-\Delta u=f(u):=u\,\left({\rm e}^{u^2}-1\right) \quad\mbox{in}\quad B_1\backslash\{0\}\\
u(|x|=1)=0,\quad u(0)=+\infty\\
u>0
\end{array}\right.$$
has infinitely many radial classical solutions. Moreover, they all satisfy
\begin{enumerate}
\item[1)] $-\Delta u=f(u)\quad\mbox{in}\quad \mathcal D^{'}(B_1).$
\item[2)] $u\in \mathcal L(B_1)$ and $\displaystyle\lim_{r\rightarrow 0}\|u\|_{\mathcal L(|x|<r)}=0.$
\end{enumerate}
\end{thm}
For the proof of Theorem \ref{Singular},  we will need the following results about the associated elliptic problem. The first is known and can be found in \cite{T} for example. The second can be seen as an extension to dimension two of  Lemma 1.1 in \cite{NiSa}.
\begin{lem}({see \cite{T}})\label{lm04}
There exists a unique radial classical solution to
$$\left\{\begin{array}{rl}
-\Delta u=f(u):=u\,\left({\rm e}^{u^2}-1\right)\,\,\mbox{in}\,\, B_1\\
u>0,\\
u=0\,\, \mbox{in}\,\, \partial B_1.
\end{array}\right.$$
\end{lem}
\begin{lem}\label{lm05}
Let $u\in C^2(B_1\backslash\{0\}),\, u\geq 0,$ such that $-\Delta u=f(u)$ in $B_1\backslash\{0\}$. Then
\begin{enumerate}
\item[i)]
$f(u)\in L^1(B_1)$.
\item[ii)]
If $(-\log(|x|))^{\alpha}u^q\in L^1(B_1),$ for some $\frac{\alpha}{q-1}>0$, then
$$\Delta u+f(u)=0\quad\mbox{in}\quad {\mathcal D}^{'}(B_1).$$
\end{enumerate}
\end{lem}
\begin{proof}[Proof of Lemma \ref{lm05}]
The proof contains two steps.
\begin{enumerate}
\item[i)]
Let $g(a):={\rm e}^{1-\frac{1}{(a-1)^2}}\chi_{[0,1[}$ and $\varphi_{\varepsilon}:x\longmapsto g(\frac{\varepsilon}{|x|})$, for $\varepsilon>0$. So
$$
\forall\; |x|\leq\varepsilon,\,\varphi_{\varepsilon}(x)=0,\,\lim_{x\neq0,\varepsilon\rightarrow 0}\varphi_{\varepsilon}(x)=1\,\mbox{ and }\,\Delta \varphi_{\varepsilon}\geq 0\;\mbox{ on }\;B_1.
$$
By Fatou's Lemma
$$
\|f(u)\|_{L^1(B_1)}\leq\liminf_{\varepsilon\rightarrow 0}\|f(u)\varphi_{\varepsilon}\|_{L^1(B_1)}
$$
Moreover
\begin{eqnarray*}
\|f(u)\varphi_{\varepsilon}\|_{L^1(B_1)}
&=&-\int_{B_1}\Delta u\varphi_{\varepsilon}(x)dx\\
&=&-\pi\Big(\int_{\varepsilon}^1\dot u(r)\varphi_{\varepsilon}({r})dr-\int_{\varepsilon}^1\dot u(r)(\varphi_{\varepsilon}(r)+r\dot{\varphi}_{\varepsilon}({r}))dr+\dot u(1)g(\varepsilon)\Big)\\
&=&-\pi\Big(\int_{\varepsilon}^1u(r)(r\ddot{\varphi}_{\varepsilon}({r})+\dot{\varphi}_{\varepsilon}({r}))dr-u(1)\dot{\varphi}_{\varepsilon}(1)+\dot u(1)g(\varepsilon)\Big)\\
&\leq&-\pi\Big(-u(1)\dot{\varphi}_{\varepsilon}(1)+\dot u(1)g(\varepsilon)\Big).
\end{eqnarray*}
Since the right hand side is bounded uniformly on $\varepsilon$, we conclude that $f(u)\in L^1(B_1)$.
\item[2i)]
Let $\tau\in C^{\infty}(\R_+)$, satisfying $\tau=0$ on $[0,1]$, $\tau =1$ on $[2,\infty)$ and $0\leq\tau\leq 1$. For $\varepsilon>0$, take $\tau_{\varepsilon}(x):=\tau(\frac{|x|}{\varepsilon})$. It is sufficient to prove that
$$\lim_{\varepsilon\rightarrow 0}\int_{B_1}\Big(u\Delta\phi+f(u)\phi\Big)\tau_{\varepsilon}=0,\quad\forall\;\phi\in{\mathcal D}(B_1).$$
Write
$$\int_{B_1}\Big(u\Delta\phi+f(u)\phi\Big)\tau_{\varepsilon}
=\int_{B_1}\Big((\Delta(\phi\tau_{\varepsilon})-2\nabla \phi\nabla \tau_{\varepsilon}-\phi\Delta\tau_{\varepsilon})u+f(u)\tau_{\varepsilon}\phi\Big),$$
so it is enough to show that
$$\lim_{\varepsilon\rightarrow 0} \int_{B_1}u\nabla \phi .\nabla \tau_{\varepsilon}=\lim_{\varepsilon\rightarrow 0} \int_{B_1}u\phi\Delta\tau_{\varepsilon}=0. $$
Let $q>1,\alpha>0$ and $\chi(x):=(-\log(|x|))^{\alpha},h:=\chi u^q.$ We have
\begin{eqnarray*}
\Big|\int_{B_1}u\phi\Delta\tau_{\varepsilon}\Big|
&\lesssim&\frac{1}{\varepsilon^2}\int_{B_{2\varepsilon}} h^{\frac{1}{q}}\chi^{\frac{-1}{q}}\\
&\lesssim&\frac{1}{\varepsilon^2}\Big(\int_{B_{2\varepsilon}} h\Big)^{\frac{1}{q}}\,\Big(\int_{B_{2\varepsilon}}\chi^{\frac{-q^{'}}{q}}\Big)^{\frac{1}{q^{'}}}\\
&\lesssim&\frac{1}{\varepsilon^{2-\frac{2}{q}}}(\int_{B_{2\varepsilon}}\chi^{\frac{-q^{'}}{q}})^{\frac{1}{q^{'}}}.
\end{eqnarray*}
In addition
\begin{eqnarray*}
\int_{B_{2\varepsilon}}\chi^{\frac{-q^{'}}{q}}
&=&\pi\int_0^{2\varepsilon}(-\log(r))^{\frac{-\alpha q^{'}}{q}}rdr\\
&=&\pi\varepsilon^2\int_0^{2}(-\log(\varepsilon r))^{\frac{-\alpha q^{'}}{q}}rdr.
\end{eqnarray*}
Since $$\lim_{\varepsilon\rightarrow 0}\int_0^{2}(-\log(\varepsilon r))^{\frac{-\alpha q^{'}}{q}}rdr=0,$$
we deduce that
$$\lim_{\varepsilon\rightarrow 0}\int_{B_1}u\phi\Delta\tau_{\varepsilon}=0.$$
Similarly, we have
$$\lim_{\varepsilon\rightarrow 0}\int_{B_1}u\nabla\phi.\nabla\tau_{\varepsilon}=0.$$
This finishes the proof.
\end{enumerate}
\end{proof}

\begin{proof}[Proof of Theorem \ref{Singular}]
First, we prove the existence of singular solutions. For any $\alpha>0$, we denote the Cauchy problem
$$(\mathcal P_{\alpha})\left\{\begin{array}{rl}
-\ddot y(t)={\rm e}^{-2t}f(y(t)),\quad t\geq 0\\
y(0)=0,\quad\dot y(0)=\alpha.
\end{array}\right.$$
Using the changing $r={\rm e}^{-t}, u_{\alpha}(x):=u_{\alpha}(|x|)=y_{\alpha}(t)$, we have classical radial solution to the elliptic problem
$$(\mathcal E_{\alpha})\left\{\begin{array}{rl}
-\Delta u_{\alpha}=f(u_{\alpha}),\quad 0<|x|\leq 1\\
u_{\alpha}(1)=0,\quad u_{\alpha}(0)=y_{\alpha}(\infty).
\end{array}\right.$$
Observe that if $\displaystyle\lim_{t\rightarrow\infty}y_{\alpha_0}(t)=l\in (0,\infty)$, for some  $\alpha_0>0$. Then $u_{\alpha_0}$ is a radial classical solution to $(\mathcal E_{\alpha_0})$ which satisfies $u_{\alpha_0}\in L^{\infty}(B)$. Then thanks to the elliptic regularity, $u_{\alpha_0}$ is necessarily a classical radial solution in the unit ball. Moreover $u_{\alpha_0}>0$ on $B(r_0)$ for some $r_0>0$. Lemma \ref{lm04} guarantees the existence of a unique $\alpha_0\in\R$ such that $\displaystyle\lim_{t\rightarrow\infty}y_{\alpha_0}(t)=l\in (0,\infty).$

Let $T(\alpha)$ be the first time for which $y_\alpha$ vanishes i.e.
$$
T(\alpha):=\sup\{s\geq 0,\,\,\mbox{ s.t }\,\, y_{\alpha}>0 \,\,\mbox{ on }\,\, (0,s)\},$$
and let $I$ be the domain of $T$ i.e.
$$
 I:=\{\alpha>0,\,\,\mbox{ s.t }\,\, T(\alpha)<\infty\}.
$$
Clearly, we have $\alpha_0\in J:=\{\alpha>0, \alpha\notin I\}$. Moreover, if there is an $\alpha\in J-\{\alpha_0\}$, then the fact that $y_{\alpha}$ is positive and concave on $(0,\infty)$ would imply $\displaystyle\lim_{t\rightarrow\infty}y_{\alpha}(t)=+\infty$. Thus, it is sufficient to prove that $J-\{\alpha_0\}$ is a not an empty interval.\\
Assume that $J-\{\alpha_0\}$ is  empty. Then necessarily $I=(0,\alpha_0)\cup (\alpha_0,\infty):=I_1\cup I_2$.\\
Now since the function $T:I\rightarrow (0,\infty)$ is a continuous, then $T(I_i)$ is an interval.
Moreover, $T(I_1)\cap T(I_2)$ should be empty because $T$ it is one to one by the uniqueness of Cauchy problem for the ODE. This is absurd because for some positive real $A$ big enough, $(A,\infty)\subset T(I_1)\cap T(I_2)$.\\
Now we repeat that same argument again. Let  $\alpha_1\in J-\{\alpha_0\}$.
Assume that $J_1:=J-\{\alpha_0,\alpha_1\}$ is  empty. Then $I=(0,\alpha_0)\cup(\alpha_0,\alpha_1)\cup (\alpha_1,\infty):=I_0\cup I_1\cup I_2$.\\
Now $T:I\rightarrow (0,\infty)$ is a continuous function, so $T(I_i)$ is an interval.
Moreover, $T(I_1)\cap T(I_2)$ is empty because $T$ it is one to one by uniqueness of elliptic associated problem.
Which is absurd because for some positive real $A$ big enough, $(A,\infty)\subset T(I_0)\cap T(I_1)\cap T(I_2)$. Thus $J_1$ is not empty. \\
The same reasoning proves that $I$ is connected and so $J$ is a not empty interval of $(0,\infty)$. The proof of this part is achieved.\\
Lemma \ref{lm05} proves that the above classical solution in $B_1\backslash\{0\}$ is extendable to a distribution solution in $B_1$.
Note that using i$)$ and H\"older inequality, the condition ii$)$ in Lemma \ref{lm05} is clearly satisfied for $f(u)=u({\rm e}^{u^2}-1)$.\\

Now we prove the last part of Theorem \ref{Singular} about Orlicz properties of singular solution.
The proof contains two parts
\begin{enumerate}
\item[i)]
Since $u(0)=\infty$, there exists $r>0$ such that $|u(x)|>1$ for any $|x|<r$. Hence
$$\int_{|x|<r}({\rm e}^{u^2}-1)dx\leq \int_{|x|<r}|u|({\rm e}^{u^2}-1)dx\leq \|f(u)\|_{L^1(B_1)}.$$
Therefore ${\rm e}^{u^2}-1\in L^1(B_1)$ via the precedent Lemma, the fact that $u\in C^2(B_1\backslash\{0\})$ and $u(|x|=1)=0$. It follows that $u\in {\mathcal L}(B_1)$.
\item[ii)]
Let $0<\varepsilon\leq 1$. By H\"older inequality we have for any $r>0$
\begin{eqnarray*}
\int_{|x|<r}\Big({\rm e}^{(\frac{u}{\varepsilon})^2}-1\Big)dx
&\lesssim&r\Big[\int_{B_1}\Big({\rm e}^{2(\frac{u}{\varepsilon})^2}-1\Big)dx\Big]^{\frac{1}{2}}.
\end{eqnarray*}
Thus, there exists $r_{\varepsilon}>0$ such that for $0\leq r<r_{\varepsilon}$, we have
$$\int_{|x|<r}\Big({\rm e}^{(\frac{u}{\varepsilon})^2}-1\Big)dx\leq 1.$$
The last inequality implies in particular that
$$\|u\|_{\mathcal L(|x|<r)}\leq \varepsilon.$$
The proof of Theorem \ref{Singular} is achieved.
\end{enumerate}
\end{proof}
%%%%%%%%%%%%%%%%%%%%%%%%%%%%%%%%%%%%%%%%%%%%%%%%%%%%%%%%%%%%%%%%%%%%%%%%%%%%%%%%%%%%%%%%%%%%%%%%%%%%%%%%%%%%%%%%%%%%%%
\subsection{Construction of a bounded solution in Orlicz space}
%%%%%%%%%%%%%%%%%%%%%%%%%%%%%%%%%%%%%%%%%%%%%%%%%%%%%%%%%%%%%%%%%%%%%%%%%%%%%%%%%%%%%%%%%%%%%%%%%%%%%%%%%%%%%%%%%%%%%%%%%%%%%%%%%%
\begin{lem}\label{aj}
Let $Q$ to be a singular solution given by Theorem \ref{Singular}. Then the Cauchy problem
\begin{equation}
\label{eq1R'}
\left\{
\begin{matrix}
\partial_t u-\Delta u= f(u)\quad\mbox{in}\quad B_1\\
u(0)= Q,\,\, u(|x|=1)=0\\
\end{matrix}
\right.
\end{equation}
has a local solution $u\in L^{\infty}([0,T]; \mathcal L(B_1))$.
\end{lem}
\begin{proof}[Proof of Lemma \ref{aj}]
 For $0<R\leq 1$ we denote by $\chi_R$ the radial function $\chi_R\in C^{\infty}_0(\R^2)$ such that $\chi_R(x)=1$ for $|x|\leq R$ and $\chi_R(x)=0$ for $|x|\geq 2R$. We decompose $$Q=(1-\chi_R)Q+\chi_RQ:=Q_1+Q_2.$$
Note that from the properties of the singular solution $Q$, we have $Q_1\in (H^1_0\cap L^{\infty})(B_1)$ and $Q_2\in \mathcal L(B_1)$ with $\displaystyle\lim_{R\rightarrow 0}\|Q_2\|_{\mathcal L(B_1)}= 0$. For the rest of the proof, we fix $R>0$ such that $\|Q_2\|_{\mathcal L(B_1)}<\min\{1,\frac{1}{4\kappa}\}$, where $\kappa$ is defined in Proposition \ref{or1}.\\
First, consider the problem
\begin{equation}
\label{eq1R}
\left\{
\begin{matrix}
\partial_t v-\Delta v= f(v)\quad\mbox{in}\quad B_1\\
v(0)= Q_1,\,\, v(|x|=1)=0.\\
\end{matrix}
\right.
\end{equation}
Using the same previous arguments, we have the existence of $ T:=T_R>0$ and $v\in {\mathcal C}([0,T); L^{\infty}\cap H^1_0(B_1))$ solution to \eqref{eq1R}.\\
Consider now the perturbed problem
\begin{equation}
\label{eq2R}
\left\{
\begin{matrix}
\partial_t w-\Delta w= f(v+w)-f(v)\quad\mbox{in}\quad B_1\\
w(0)= Q_2,\,\, w(|x|=1)=0,\\
\end{matrix}
\right.
\end{equation}
and denote by $$X_T:=L^{\infty}([0,T];\mathcal L(B_1)),\quad\mbox{endowed with the norm}\quad \|\,.\,\|_{T}:=\|\,.\,\|_{L^{\infty}([0,T],\mathcal L(B_1))}.$$
Define $w_l:={\rm e}^{t\Delta}Q_2$ and consider the map
$$\Phi : w\mapsto \tilde w:=\int_0^t{\rm e}^{(t-s)\Delta}(f(v+w+w_l)-f(v))(s)ds.$$
We prove that $\Phi$ is a contraction in the ball $B_T(r)$ of $X_T$ for some small $r,\;T>0$.\\
Let $w_1,w_2\in B_T(r)$ and set $w:=w_1-w_2$ and $u_i:=v+w_i+w_l, i\in \{1,2\}$. By the smoothing effect \eqref{smoth}, we have for any $\varepsilon>0$
\begin{equation}\label{sm}
\|\tilde w_1-\tilde w_2\|_{L^1}\lesssim \int_0^t\|f(u_1)-f(u_2)\|_{L^1}ds,\quad\|\tilde w_1-\tilde w_2\|_{L^{\infty}}\lesssim\int_0^t \frac{\|f(u_1)-f(u_2)\|_{L^{1+\varepsilon}}}{(t-s)^{\frac{1}{1+\varepsilon}}}ds.
\end{equation}
Now, using a convexity argument,  H\"older inequality and Proposition \ref{or1}, we have for small $\varepsilon>0$,
\begin{eqnarray*}
\|f(u_1)-f(u_2)\|_{L^{1+\varepsilon}}
&\lesssim&\sum_{i=1}^2\|wu_i^2{\rm e}^{u_i^2}\|_{L^{1+\varepsilon}}\\
&\lesssim&{\rm e}^{2\|v\|_{L^{\infty}}^2}\sum_{i=1}^2\|wu_i^2{\rm e}^{2(w_i+w_l)^2}\|_{L^{1+\varepsilon}}\\
&\lesssim&{\rm e}^{2\|v\|_{L^{\infty}}^2}\sum_{i=1}^2(\|wu_i^2\|_{L^{1+\varepsilon}}+\|w({\rm e}^{2(w_i+w_l)^2}-1)\|_{L^{1+\varepsilon}})\\
&\lesssim&{\rm e}^{2\|v\|_{L^{\infty}}^2}\|w\|_{\mathcal L}\sum_{i=1}^2(\|u_i^2\|_{L^{1+2\varepsilon}}+\|{\rm e}^{2(w_i+w_l)^2}-1\|_{L^{1+2\varepsilon}})\\
&\lesssim&{\rm e}^{2\|v\|_{L^{\infty}}^2}\|w\|_{\mathcal L}\sum_{i=1}^2(r^2+\|v\|_{H^1}^2+\|w_l\|_{\mathcal L}^2+\|{\rm e}^{2(w_i+w_l)^2}-1\|_{L^{1+2\varepsilon}}).
\end{eqnarray*}
Moreover, we have
\begin{eqnarray*}
\int_{\R^2}({\rm e}^{ 2(1+2\varepsilon)(w_1+w_l)^2}-1)dx
&\leq&\int_{\R^2}({\rm e}^{4(1+2\varepsilon)[w_1^2+w_l^2]}-1)dx\\
&\leq&\int_{\R^2}({\rm e}^{ 4(1+2\varepsilon)w_1^2}-1)({\rm e}^{4(1+2\varepsilon)w_l^2}-1)dx\\
&+&\int_{\R^2}({\rm e}^{4(1+2\varepsilon)w_1^2}-1)dx+\int_{\R^2}({\rm e}^{4(1+2\varepsilon)w_l^2}-1)dx.
\end{eqnarray*}
Taking $r>0$ small enough, we have, by H\"older inequality and Proposition \ref{or1},
 $$\int_{\R^2}({\rm e}^{2(1+2\varepsilon)(w_1+w_l)^2}-1)dx\lesssim 1+\int_{\R^2}({\rm e}^{4(1+4\varepsilon)w_l^2}-1)dx.$$
By $L^{\infty}-L^1$ interpolation (see Proposition \ref{or1}), we have $\|w_l\|_{\mathcal L}\leq \kappa \|Q_2\|_{\mathcal L}<\frac{1}{4}$. Hence, for $\varepsilon>0$ small enough we have $ \int_{\R^2}({\rm e}^{4(1+4\varepsilon)w_l^2}-1)dx\leq 1$. Therefore
$$\|f(u_1)-f(u_2)\|_{L^{1+\varepsilon}}\leq C_{r,R}\|w\|_T.$$
Using \eqref{sm}, we have for small $T>0$,
$$\|\tilde w_1-\tilde w_2\|_{T}\lesssim T\|w\|_T.$$
Moreover, taking in the last inequality $w_2=0$,
$$\|\tilde w_1\|_{T}\lesssim Tr+\|\int_0^t{\rm e}^{(t-s)\Delta}(f(w_l+v)-f(v))ds\|_T.$$
 Using similar arguments we clearly have
 $$
 \|\int_0^t{\rm e}^{(t-s)\Delta}(f(w_l+v)-f(v))ds\|_T\lesssim T\,.
 $$
All this shows that $\Phi$ is a contraction of $B_T(r)$ for small positive $r,T$. Let $w$ be the fixed point of $\Phi$ in $B_T(r)$, and $u:=v+w+w_l$. Clearly $u$ is the desired solution of Problem \eqref{eq1R'} in $L^{\infty}([0,T]; \mathcal L(B_1))$.
\end{proof}
%%%%%%%%%%%%%%%%%%%%%%%%%%%%%%%%%%%%%%%%%%%%%%%%%%%%%%%%%%%%%%%%%%%%%%%%%%%%%%%%%%%%%%%%%%%%%%%%%%%%%%%%%
\subsection{Non uniqueness}
%%%%%%%%%%%%%%%%%%%%%%%%%%%%%%%%%%%%%%%%%%%%%%%%%%%%%%%%%%%%%%%%%%%%%%%%%%%%%%%%
In this subsection we construct a non stationary solution $u$ to the Cauchy problem \eqref{eq2} in $\mathcal L(B_1)$ with initial data $Q$. We show that its potential term will be slightly better than $L^1$. This will suffices to apply Brezis-Cazenave result about regularization effect of the heat equation. Thus, the solution satisfies $u\in L^{\infty}((0,T);L^{\infty})$.
\begin{prop}\label{smm}
Let $u_0\in \mathcal L(B_1)$ and $u\in L^{\infty}((0,T);\mathcal L(B_1))$ solution to \eqref{eq2} with data $u_0$ such that $\|u\|_{L^{\infty}([0,T];\mathcal L(B_1))}\leq \frac{1}{1+\varepsilon}$ for some $\varepsilon>0$. Then
$$u\in L^{\infty}([0,T];L^{\infty}).$$
\end{prop}
\begin{proof}
It is clear that $\|{\rm e}^{u^2}-1\|_{L^{\infty}((0,T),L^{1+\varepsilon})}\leq 1$ because
$$\int_{\R^2}({\rm e}^{(1+\varepsilon)u^2}-1)dx\leq\int_{\R^2}({\rm e}^{(\frac{u}{\|u\|_T})^2}-1)dx\leq 1.$$
Now, applying Theorem A.1 in \cite{Br} (see Theorem \ref{BrCaz}) in choosing $\sigma =1+\varepsilon$, $r\geq 1+\frac{1}{\varepsilon}$ and $a:={\rm e}^{u^2}-1$, there exists a unique $v\in L^{\infty}([0,T]; L^r)$ solution to \eqref{eq2} with data $u_0$. Since ${\mathcal C}([0,T];\mathcal L)\subset L^{\infty}([0,T]; L^r)$, then $u=v$.
\end{proof}
Now we prove the nonuniqueness result given by Theorem \ref{MainOrlicz}.
\begin{proof}[Proof of Theorem \ref{MainOrlicz}]
Let $Q$, given by Theorem \ref{Singular} to be a singular solution to the elliptic stationary problem associated to \eqref{eq2}. Then, by Lemma \ref{aj} there exists $u\in L^{\infty}([0,T]; \mathcal L(B_1))$ a solution to \eqref{eq1R'}. Moreover, applying Proposition \ref{smm} to the second part of $u$ (see decomposition of $u$ in the proof of Lemma \ref{aj}) which is solution to \eqref{eq2R}, we conclude that $u\in L^{\infty}([0,T];L^{\infty}(B_1))$.
Then, the problem \eqref{eq1R'}
has two different solutions, $v(t):=Q\in \mathcal C([0,\infty);\mathcal L)$ and another solution $u\in L^{\infty}([0,T];L^{\infty})$. But $\displaystyle\lim_{x\rightarrow 0}v(x)=\infty$, thus
$$u\neq v.$$
\end{proof}
%%%%%%%%%%%%%%%%%%%%%%%%%%%%%%%%%%%%%%%%%%%%%%%%%%%%%%%%%%%%%

%%%%%%%%%%%%%%%%%%%%%%%%%%%%%%%%%%%%%%%%%%%%%%%%%%%%%%%%%%%%%


\begin{thebibliography}{99}

%@@@@@@@@@@@@@@@@@@@@@@@@@@@@@@@@@@@@%@@@@@@@@@@@@@@@@@@@@@@@@@@@@@@@@@@@@%@@@@@@@@@@@@@@%@@@@@@@@@@@@@@@@@@@@@@@@@@@@@@@@@@@@%@@@@@@@@@@@@@@@@@@@@@@@@@@@@@@@@@@@@%@@@@@@@@@@@@@@


\bibitem{Ad}{\bf S. Adachi and K. Tanaka}, {\em Trudinger type inequalities in ${\mathbb{R}}^{N}$ and their best exponent}, Proc. Amer. Math. Society {128}  (1999), N.7. , pp. 2051--2057.

\bibitem{HMN}{\bf H. Bahouri, M. Majdoub and N. Masmoudi}, {\em On the lack of compactness in the $2D$ critical Sobolev embedding}, submitted.


\bibitem{Br}{\bf H. Brezis, T. Casenave}, {\em A nonlinear heat equation with singular initial data}, Journ. d'anal. math, 68. (1996), pp. 73--90.

\bibitem{cv}{\bf L. Caffarelli, A. Vasseur}, {\em Drift diffusion equations with fractional diffusion
and the quasi-geostrophic equation}, Annals of Mathematics, Vol. 171 (2010), No. 3, pp. 1903-1930

\bibitem{Cas}{\bf T. Cazenave}, {\em An introduction to nonlinear Schr\"odinger equations}, Textos de Metodos Matematicos {26}, Instituto de Matematica UFRJ, (1996).


\bibitem{Col.I}{\bf J. Colliander, S. Ibrahim, M. Majdoub and N. Masmoudi}, {\em Energy critical NLS in two space dimensions}, Journal of Hyperbolic Differential Equations, Vol. {\bf 6} (2009), pp. 549--575.


\bibitem{HW}{\bf A. Haraux and F. B. Weissler}, {\em Non uniqueness for a semilinear initial value problem}, , Indiana Univ. Math. J. 31  (1982), pp. 167--189,.


\bibitem{I}{\bf S. Ibrahim, M. Majdoub, N. Masmoudi, and K. Nakanishi}, {\em Scattering for the two-dimensional energy-critical wave equation}, Duke Mathematical Journal, Vol. {\bf 150} (2009), pp. 287--329.

\bibitem{Ib1}{\bf S. Ibrahim, M. Majdoub and N. Masmoudi}, {\em Double logarithmic inequality with a sharp constant}, Proc. Amer. Math. Soc. 135, no. 1 (2007), pp. 87--97.


\bibitem{Ib2}{\bf S. Ibrahim, M. Majdoub and N. Masmoudi}, {\em Global solutions for a semilinear 2D Klein-Gordon equation with exponential type nonlinearity}, Comm. Pure App. Math, Volume 59, Issue 11 (2006), pp. 1639--1658.

\bibitem{Ib3}{\bf S. Ibrahim, M. Majdoub and N. Masmoudi}, {\em Instability of $H^1$-supercritical waves}, C. R. Acad. Sci. Paris, ser. I 345 (2007), pp. 133--138.


\bibitem{Ib4}{\bf S. Ibrahim, M. Majdoub and N. Masmoudi}, {\em Well and ill-posedness issues for energy supercritical waves}, To appear in Analysis \& PDE.


\bibitem{IMN}
{\bf  S.~Ibrahim, N.~Masmoudi and K.~Nakanishi}, {\em Scattering threshold for the focusing nonlinear Klein-Gordon equation.} To appear in Analysis and PDE.


\bibitem{L}
{\bf H. A. Levine}, {\em Some nonexistence and stability theorems for solutions of formally parabolic equations of the form $Pu_t=-Au+F(u)$}, Arch. Rational Mech. Anal. 51 (1973), 371–386.


\bibitem{MaTerr} {\bf J. Matos and E. Terraneo}, {\em Nonuniqueness for a critical nonlinear heat equation with any
              initial data}, {Nonlinear Anal.}, {\bf 55} (2003), pp. 927--936.

\bibitem{MI}{\bf C. Miao and B. Zhang}, {\em The Cauchy problem for semilinear parabolic equations in Besov spaces},  Houston J. Math.  {\bf 30} (2004),  no. 3, pp. 829--878.

\bibitem{Mo}{\bf J. Moser}, {\em A sharp form of an inequality of N. Trudinger}, Ind. Univ. Math. J. 20 (1971), pp. 1077--1092.

\bibitem{NiSa} {\bf W.-M. Ni and P. Sacks}, {\em Singular behavior in nonlinear parabolic equations}, {Trans. Amer. Math. Soc.}, {\bf 287} (1985), pp. 657--671.
		


\bibitem{Rib}{\bf F. Ribaud}, {\em Cauchy problem for semilinear parabolic equations with initial
              data in {$H^s_p({\bf R}^n)$} spaces}, Rev. Mat. Iberoamericana, 14 (1998), pp. 1--46.


\bibitem{Ru}{\bf B. Ruf}, {\em A sharp Moser-Trudinger type inequality for unbounded domains
in ${\mathbb{R}}^{2}$},  J. Funct. Analysis,  219 (2004), pp. 340--367.

\bibitem{RT}{\bf B. Ruf and E. Terraneo}, {\em The Cauchy problem for a semilinear heat equation with singular initial data}, Evolution equations, semigroups and functional analysis ({M}ilano, 2000), Progr. Nonlinear Differential Equations Appl.,Vol. 50, pp. 295--309, Birkh\"auser, 2002.

\bibitem{T}{\bf C. Tarsi}, {\em Uniqueness of positive solutions of nonlinear ellipyic equations with exponential growth}, Proceeding of the Royal Society of Edinburgh, 133A (2003), pp. 1409--1420.


\bibitem{Tan} {\bf Z. Tan}, {\em Global solution and blowup of semilinear heat equation with critical {S}obolev exponent}, {Comm. Partial Differential Equations}, {\bf 26} (2001), pp. 717--741.
		

\bibitem{Tay}{\bf M. E. Taylor}, {\em Partial differential equations. III. Nonlinear equations. Corrected reprint of the 1996 original}, Springer-Verlag, New York, 1997. xxii+608 pp.


\bibitem{Terr} {\bf E. Terraneo}, {\em Non-uniqueness for a critical non-linear heat equation}, {Comm. Partial Differential Equations}, {\bf 27} (2002), pp. 185--218.
		

\bibitem{Tr}{\bf N. S. Trudinger}, {\em On imbedding into Orlicz spaces and some applications}, J. Math. Mech. 17 (1967), pp. 473--484.



\bibitem{Ws1}{\bf F. B. Weissler}, {\em Local existence and nonexistence for a semilinear parabolic equation in $L^p$}, Indiana Univ. Math. J. 29 (1980), pp. 79--102.

\bibitem{Ws2}{\bf F. B. Weissler}, {\em Existence and nonexistence of global solutions for a semilinear heat equation}, Israel J. Math., 38 (1981), pp. 29--40.
%@@@@@@@@@@@@@@@@@@@@@@@@@@@@@@@@@@@@%@@@@@@@@@@@@@@@@@@@@@@@@@@@@@@@@@@@@%@@@@@@@@@@@@@@%@@@@@@@@@@@@@@@@@@@@@@@@@@@@@@@@@@@@%@@@@@@@@@@@@@@@@@@@@@@@@@@@@@@@@@@@@%@@@@@@@@@@@@@@


\end{thebibliography}
\end{document}